\documentclass[10pt]{amsart}
\usepackage{amsmath}
\usepackage{amssymb}
\usepackage{mathtools}      
\usepackage{mathabx} 
\usepackage{enumitem}
\usepackage{amsthm}
\usepackage{thmtools}
\usepackage{tikz,pgfplots}
\pgfplotsset{compat=1.18} 
\usetikzlibrary{decorations.text}
\usetikzlibrary{arrows.meta}
\usepackage[font=small]{caption}
\usepackage[colorlinks=true,backref=page,bookmarksopen=true]{hyperref} 
\usepackage[capitalise]{cleveref} 

\usepackage{xcolor}
\hypersetup{
    colorlinks,
    linkcolor={red!50!black},
    citecolor={blue!50!black},
    urlcolor={blue!80!black}
}

\renewcommand*{\backref}[1]{}
\renewcommand*{\backrefalt}[4]{\quad \tiny 
  \ifcase #1 ({\color{red} \bf NOT CITED.})%
  \or    (Cited on p.~#2.)%
  \else   (Cited on pp.~#2.)%
  \fi}
  
\declaretheorem[numberwithin=section]{theorem} 
\declaretheorem[sibling=theorem]{proposition} 

\declaretheorem[sibling=theorem]{lemma}

\declaretheorem[sibling=theorem]{question}
\declaretheorem[sibling=theorem, style=remark]{remark}

\declaretheorem[sibling=theorem, style=definition]{definition}

\hypersetup{bookmarksdepth = 3} 
\numberwithin{equation}{section}

\setlist[enumerate,1]{label={\upshape \arabic*.},ref=\arabic*}

\newcommand{\N}{\mathbb{N}} \newcommand{\Z}{\mathbb{Z}}  \newcommand{\R}{\mathbb{R}} 

\newcommand{\st}{\;\mathord{:}\;}

\DeclareMathOperator{\diam}{diam}

\renewcommand{\setminus}{\smallsetminus}
\renewcommand{\emptyset}{\varnothing}

\newcommand{\arxiv}[2]{\href{http://arxiv.org/abs/#1}{arXiv: {#1} [{#2}]}}
\newcommand{\doi}[1]{\href{http://doi.org/#1}{\tt doi}} 
\newcommand{\directlink}[1]{\href{#1}{\tt URL}}
\newcommand{\MRev}[1]{\href{https://mathscinet.ams.org/mathscinet-getitem?mr=#1}{\tt MR}} 
\newcommand{\Zbl}[1]{\href{https://zbmath.org/?q=an:#1}{\tt Zbl}} 

\DeclareFontFamily{U} {MnSymbolA}{} 
\DeclareFontShape{U}{MnSymbolA}{m}{n}{
   <-6> MnSymbolA5
   <6-7> MnSymbolA6
   <7-8> MnSymbolA7
   <8-9> MnSymbolA8
   <9-10> MnSymbolA9
   <10-12> MnSymbolA10
   <12-> MnSymbolA12}{}
\DeclareFontShape{U}{MnSymbolA}{b}{n}{
   <-6> MnSymbolA-Bold5
   <6-7> MnSymbolA-Bold6
   <7-8> MnSymbolA-Bold7
   <8-9> MnSymbolA-Bold8
   <9-10> MnSymbolA-Bold9
   <10-12> MnSymbolA-Bold10
   <12-> MnSymbolA-Bold12}{}
\DeclareSymbolFont{MnSyA} {U} {MnSymbolA}{m}{n}
\DeclareFontFamily{U} {MnSymbolC}{}
\DeclareFontShape{U}{MnSymbolC}{m}{n}{
  <-6> MnSymbolC5
  <6-7> MnSymbolC6
  <7-8> MnSymbolC7
  <8-9> MnSymbolC8
  <9-10> MnSymbolC9
  <10-12> MnSymbolC10
  <12-> MnSymbolC12}{}
\DeclareFontShape{U}{MnSymbolC}{b}{n}{
  <-6> MnSymbolC-Bold5
  <6-7> MnSymbolC-Bold6
  <7-8> MnSymbolC-Bold7
  <8-9> MnSymbolC-Bold8
  <9-10> MnSymbolC-Bold9
  <10-12> MnSymbolC-Bold10
  <12-> MnSymbolC-Bold12}{}
\DeclareSymbolFont{MnSyC} {U} {MnSymbolC}{m}{n}

\DeclareMathSymbol{\smallplus}{\mathord}{MnSyC}{20} 
\DeclareMathSymbol{\smallpm}{\mathord}{MnSyC}{22} 

\begin{document}

\title{Monochromatic nonuniform hyperbolicity}
\author{Jairo Bochi}
\address{Department of Mathematics, The Pennsylvania State University}
\email{\href{mailto:bochi@psu.edu}{bochi@psu.edu}}
\date{September 12, 2025}

\begin{abstract}
We construct examples of continuous $2$-dimensional linear cocycles which are not uniformly hyperbolic despite having the same non-zero Lyapunov exponents with respect to all invariant measures. The base dynamics can be any non-trivial subshift of finite type. According to a theorem of DeWitt--Gogolev and Guysinsky, such cocycles cannot be H\"older-continuous. Our construction uses the nonuniformly hyperbolic cocycles discovered by Walters in 1984.
\end{abstract}

\maketitle

\section{Introduction}

\subsection{Regularity thresholds}

Qualitative and quantitative properties of dynamical systems depend crucially on regularity, and  rough systems may display peculiar behavior. 
For example, as discovered by Denjoy almost a hundred years ago, there exist non-minimal circle $C^1$-diffeomorphisms without periodic points, but such examples cannot be made of class $C^2$ (see e.g.\ \cite[Chap.~12]{KatokH}).

On the other hand, several perturbative techniques, such as the Closing Lemma, are known only in low regularity. As a result, the goal of describing properties of ``typical'' or ``generic'' dynamical systems becomes harder if restricted to high regularity. 

H\"older continuity is a natural and frequent regularity assumption in hyperbolic dynamics: see \cite[{\S}6.5.e]{HKatok} for discussion. 
Exotic phenomena can be found below the H\"older threshold: see \cite{Bowen,RobY,Pugh,Quas,AvilaB,JenkMorris,BCS,Kocsard,BMOS,Kosloff,Bochi_isolated} for examples.

This paper is motivated by a result of DeWitt and Gogolev \cite{DWG} and Guysinsky \cite{Guy}. Their theorem concerns H\"older-continuous linear cocycles over hyperbolic base dynamics and provides a strong and uniform conclusion (existence of a dominated splitting) based on a relatively weak and nonuniform hypothesis (narrow Lyapunov spectrum). In this paper, we construct examples showing that the H\"older regularity assumption cannot be dropped in the DeWitt-Gogolev-Guysinsky result.

\subsection{Setting and terminology}
Suppose that $X$ is a compact metric space and $T \colon X \to X$ is a homeomorphism.
If $G$ is a topological group and $A \colon X \to G$ is a continuous map,
then the pair $(T,A)$ is called a \emph{continuous $G$-cocycle}, or simply, a \emph{$G$-cocycle}.
Associated to the pair $(T,A)$ we have a unique map $(x,n) \in X \times \Z \mapsto  A^{(n)}(x) \in G$ 
that satisfies the following conditions:
\begin{equation}\label{e.cocycle_eq}
A^{(1)}(x) = A(x) \, , \quad 
A^{(m+n)}(x) = A^{(m)}(T^n x) A^{(n)}(x) \, .
\end{equation}
Some authors call this map a cocycle (see \cite[{\S}1.3.k]{HKatok}).
When $G$ is a matrix group, we speak of \emph{linear cocycles}.

Let us recall some basic facts about linear cocycles.
We restrict our discussion to the case where $G$ is the group $\mathrm{SL}^\smallpm(2,\R)$ of 
real $2 \times 2$ matrices with determinant~$\pm 1$. 
Proofs of all facts stated below can be found for example in \cite[\S~3.8]{DamanikF}. 

Let $(T,A)$ be $\mathrm{SL}^\smallpm(2,\R)$-cocycle.
Let $\mu$ be a Borel probability measure on $X$ which is invariant and ergodic with respect to $T$.
Then there exists a nonnegative number $\lambda(T,A,\mu)$, called the \emph{(top) Lyapunov exponent} of the cocycle $(T,A)$ with respect to the measure $\mu$, such that
\begin{equation}
\lim_{n \to \pm \infty} \frac{1}{n} \log \|A^{(n)}(x) \| = \lambda(T,A,\mu)  \quad \text{for $\mu$-almost every $x$} \, .
\end{equation}
The \emph{Lyapunov spectrum} of the cocycle is defined as the set 
\begin{equation}\label{e.spec}
	\Lambda(T,A) \coloneqq \big\{\lambda(T,A,\mu) \st \mu \text{ is ergodic} \big\} \, .
\end{equation}

An ergodic measure is called \emph{hyperbolic} if $\lambda(T,A,\mu)>0$.
In that case, for $\mu$-almost every $x$, there exists a unique \emph{Oseledets splitting} of $\R^2$ as a sum $E^\mathrm{u}(x) \oplus E^\mathrm{s}(x)$ of one-dimensional spaces such that $\frac{1}{n} \log \|A^{(n)}(x) v \|$ converges as $n \to \pm \infty$ to $\lambda(T,A,\mu)$ for all nonzero vectors $v$ in $E^\mathrm{u}(x)$, and to $-\lambda(T,A,\mu)$ for all nonzero vectors in $E^\mathrm{s}(x)$.
The number $-\lambda(T,A,\mu)$ is the \emph{bottom Lyapunov exponent}.
The Oseledets splitting depends measurably on the point, and is invariant in the sense that 
\begin{equation}
A(x) (E^\mathrm{u}(x)) = E^\mathrm{u}(Tx) \, , \quad
A(x) (E^\mathrm{s}(x)) = E^\mathrm{s}(Tx) \, .
\end{equation}
The original reference is \cite{Oseledets}.

We say that the cocycle $(T,A)$ is \emph{uniformly hyperbolic} if the norms $\|A^{(n)}(x) \|$ grow exponentially in a uniform way, that is,
\begin{equation}
\liminf_{n \to \pm \infty} \inf_{x\in X} \frac{1}{|n|} \log \|A^{(n)}(x) \| > 0 \, .
\end{equation}
In that case, for every point $x\in X$, there exists a \emph{hyperbolic splitting} $\R^2 = E^\mathrm{u}(x) \oplus E^\mathrm{s}(x)$ such that
vectors in $E^\mathrm{s}(x)$ (resp.\ $E^\mathrm{u}(x)$) are uniformly contracted under positive (resp.\ negative) iterates, that is, 
\begin{equation}
\limsup_{n \to + \infty} \sup_{x \in X} \frac{1}{n} 
\log \max \left\{ 
\big\| A^{(n)}(x)|_{E^\mathrm{s}(x)} \big\| , \ 
\big\| A^{(-n)}(x)|_{E^\mathrm{u}(x)} \big\|
\right\} 
< 0 \, .
\end{equation}
The hyperbolic splitting is continuous and invariant.
Furthermore, it coincides $\mu$-almost everywhere with the Oseledets splitting, for any ergodic probability measure~$\mu$.

\medskip

We remark that all concepts introduced so far make sense in arbitrary dimension, but we will refrain from providing details. 
In general, there are several Lyapunov exponents, and the measure is called \emph{hyperbolic} if none of them is zero. Such measures play an important role in differentiable dynamics: see \cite{BP}. Uniform hyperbolicity was introduced in differentiable dynamics by Anosov and Smale in the 1960s.
A uniformly hyperbolic splitting is always \emph{dominated} in the sense of \cite[{\S}B.1]{BDV}, and the converse is true for $\mathrm{SL}^\smallpm(2,\R)$-cocycles. In ODE theory, uniform hyperbolicity and dominated splittings are sometimes called \emph{exponential dichotomy} \cite[{\S}1.4]{JohnEtc} and \emph{exponential separation} \cite[p.~189]{ColoniusK}, respectively. Some types of Lyapunov spectra are studied in the papers \cite{Bochi_ICM,Park,BSert}.

\subsection{The main result}

Back to $\mathrm{SL}^\smallpm(2,\R)$-cocycles, note that if $(T,A)$ is uniformly hyperbolic, then all its ergodic measures are hyperbolic; actually, the Lyapunov spectrum \eqref{e.spec} is away from zero, i.e., $\inf \Lambda(T,A) > 0$.
Consider the converse statement:
\begin{equation}\label{e.converse}
\inf \Lambda(T,A) > 0 
\quad \overset{?}{\Rightarrow} \quad 
(T,A) \text{ uniformly hyperbolic} \, ;
\end{equation}
in other words, does ``uniform nonuniform hyperbolicity'' imply uniform hyperbolicity?
The answer is negative since there exist $\mathrm{SL}^\smallpm(2,\R)$-cocycles over uniquely ergodic base dynamics which are not uniformly hyperbolic despite having a positive Lyapunov exponent with respect to the unique invariant probability measure. Herman \cite{Herman} provides several examples of such cocycles. Another class of examples was constructed by Walters \cite{Walters}, and we will discuss it in detail later (see \cref{ss.Walters}). Actually, the very first examples of (continuous-time) nonuniform hyperbolicity with uniquely ergodic base were exhibited in the Soviet literature: see \cite[{\S}8.7]{JohnEtc} for references and a modern exposition.

\medskip

Let us specialize further to the situation where $T \colon X \to X$ is a transitive subshift of finite type (SFT), and therefore (as long as we exclude the uninteresting case of finite state space), far from being uniquely ergodic.
Then statement~\eqref{e.converse} holds for locally constant cocycle maps $A$.
More generally, it holds if the cocycle map $A$ is H\"older continuous and fiber-bunched: see \cite[Theorem~1.5]{Velozo}. 
However, statement~\eqref{e.converse} 
does not hold for all H\"older cocycles, as shown by an example of Velozo Ruiz \cite[Theorem~4.1]{Velozo} (preceded by constructions of \cite{CLR,Gogolev} in smooth dynamics).
Let us note that, even though the choice of metric in $X$ is not canonical, the class of H\"older maps is: see \cite[{\S}1.9.a]{KatokH}).

A result of DeWitt and Gogolev \cite[Theorem~1.3]{DWG} (see also Guysinsky \cite{Guy}), if specialized to dimension~$2$, provides other sufficient conditions for the validity of statement~\eqref{e.converse}.

\begin{definition}
We say that a continuous $\mathrm{SL}^\smallpm(2,\R)$-cocycle $(T,A)$ is \emph{monochromatic} if its Lyapunov spectrum $\Lambda(T,A)$ is a singleton $\{\lambda_0\}$. Equivalently, all ergodic Borel probability measures have same Lyapunov exponent $\lambda_0$.
\end{definition}

\begin{theorem}[DeWitt and Gogolev, Guysinsky] \label{t.DGG}
Let $T$ be a transitive SFT and let $A \colon X \to \mathrm{SL}^\smallpm(2,\R)$ be a H\"older map.
If the cocycle $(T,A)$ is monochromatic with a positive Lyapunov exponent, then it is uniformly hyperbolic.
\end{theorem}

In order to apply the theorem above, it is actually sufficient to check that all invariant measures supported on periodic orbits have the same Lyapunov exponent -- this follows from Kalinin's approximation theorem \cite[Theorem~1.4]{Kalinin}, which valid in the SFT / H\"older setting (but not in general, see \cite{Bochi_isolated}).  
Actually, \cref{t.DGG} is still valid if the Lyapunov spectrum $\Lambda(T,A)$ is ``narrow'', that is, sufficiently close to a point: see \cite[Theorem~1.3]{DWG}. 

In this paper, we prove that the H\"older hypothesis is indispensable in \cref{t.DGG}:

\begin{theorem}\label{t.main_SFT}
Let $T \colon X \to X$ be a subshift of finite type, where $X$ is not a finite set. 
Fix a constant $\lambda_0>0$.
Then there exists a continuous map $A \colon X \to \mathrm{SL}^{\smallpm}(2,\R)$ such that 
the cocycle $(T,A)$ is monochromatic with Lyapunov exponent $\lambda_0$, but it is not uniformly hyperbolic. 
\end{theorem}

Actually, the conclusion holds not only for SFTs, but also for essentially any base dynamics that contains a nontrivial mixing SFT as a subsystem: see \cref{t.main_embed}.

\subsection{Localized nonuniform hyperbolicity}

In our examples, the obstruction to uniform hyperbolicity is ``localized'' in the following sense:
there exists a $T$-invariant compact set $Y \subset X$ such that $T|_Y$ is transitive and uniquely ergodic, the restricted cocycle $(T|_Y,A|_Y)$ is not uniformly hyperbolic, and for every compact invariant set $Z \subset X \setminus Y$, the restricted cocycle $(T|_Z,A|_Z)$ is uniformly hyperbolic. 
Let us note that the aforementioned example of Velozo Ruiz \cite[Theorem~4.1]{Velozo} has the same localization property. However, in his case the set $Y$ is relatively simple, namely the closure of an orbit which is homoclinic to a fixed point, while in our situation the set $Y$ is rather complicated: the restricted cocycle $(T|_Y,A|_Y)$ is of Walters type.
The Katok area-preserving diffeomorphism \cite{Katok} is another example of ``localized nonuniform hyperbolicity'' and in this case the set $Y$ is finite.  

All those examples (including ours) have the common feature that hyperbolicity degenerates in a well-controlled way as one approaches the set $Y$. The main difficulty which the present work overcomes is allowing hyperbolicity to degenerate while keeping the Lyapunov spectrum monochromatic. 

As outlined previously, our construction uses a cocycle of Walters type as an obstruction to uniform hyperbolicity. We leave the definition of Walters cocycles for \cref{ss.Walters}, but let us mention the feature that sets them apart from other examples of nonuniformly hyperbolic cocycles: Walters cocycles are formed by entrywise \emph{nonnegative} matrices. This is the key property that makes our proof possible. An informal explanation of the relevance of this property is given at \cref{ss.positive_summary}.

\subsection{Alternative methods?}

At this point the reader may wonder whether Velozo Ruiz's relatively simple construction can be adapted to a proof of \cref{t.main_SFT}. Let us recall some additional features of the example from \cite{Velozo}: the base dynamics is the full $2$-shift and the matrix function is of the form $A(x) = A_0 R_{\theta(x)}$, where $A_0$ is the constant matrix $\mathrm{diag}(2,1/2)$ and $R_\theta$ denotes the rotation matrix by angle $\theta$. The angle function $\theta$ is carefully chosen: among other conditions, it is nonnegative, supported on the neighborhood of a point $q$ which is homoclinic to a fixed point $p$, and it attains its maximum $\pi/2$ at the point $q$. Since the right-angle rotation interchanges the two invariant directions of the original cocycle $A_0$, the resulting cocycle cannot be uniform hyperbolic. Nevertheless, with an appropriate choice of the function $\theta$, it can be shown that orbits passing near -- but not through -- $q$ have sufficient time to ``recover'' from the rotation, leading to Lyapunov exponents that are uniformly bounded away from zero. 

The example is therefore a deformation of a constant cocycle $A_0$, whose Lyapunov exponent is $\log 2$. 
While it is clear that proximity to $q$ results in a controlled loss of Lyapunov exponent, one might ask whether this loss could be compensated by further perturbations away from $q$. Consider periodic points near $q$ whose orbits remain most if the time close to $p$. Any attempt at compensation for such a periodic orbit would need to occur near $p$, where we wish to preserve the original Lyapunov exponent $\log 2$. These two goals seem incompatible. For this reason, the method does not appear to offer a promising route toward the construction of examples like ours.

\medskip
\noindent\textbf{Notation:} Integer intervals are denoted using double brackets, that is, if $m,n \in \Z$, then $\ldbrack m,n \rdbrack \coloneqq \{ k \in \Z \st m \le k \le n\}$.

\section{Scheme of the proof}

\subsection{Veech-like maps}

Walters' construction requires special properties of the base dynamics, which are encapsulated in the following definition: 

\begin{definition}
Let $Y$ be a compact metric space and let $T \colon Y \to Y$ be a homeomorphism. We say that $T$ is a \emph{Veech-like map} if $T$ is minimal and uniquely ergodic and $T^2$ is minimal but not uniquely ergodic. 
\end{definition}

Veech-like maps were first constructed long ago by Veech in the study of a number-theoretical problem \cite{Veech69}. His examples were later identified as being isomorphic to interval exchange transformations (IET): see \cite{KeynesNewton}, \cite[p.~798]{Veech77}. For a modern perspective and variations of Veech's construction, see \cite{Ferenczi}. 

We will use shift maps that are Veech-like:

\begin{theorem} \label{t.existence_Veech}
There exists a shift space on finitely many symbols such that the corresponding shift transformation is a Veech-like map and has zero topological entropy. 
\end{theorem}

In order to keep this paper self-contained, we will present a complete proof of the statement above in \cref{s.symbolic_Veech}.

The next lemma describes some basic properties of Veech-like maps:

\begin{lemma}\label{l.double_ergodicity}
If $T \colon Y \to Y$ is a Veech-like map, then $T^2$ admits exactly two ergodic probability measures $\nu_0$ and $\nu_1$, both of which have support $Y$. Furthermore, $T_*(\nu_0) = \nu_1$, $T_*(\nu_1) = \nu_0$, and the unique $T$-invariant probability measure is $\nu = \frac{\nu_0+\nu_1}{2}$. 
\end{lemma}	
	
Note that the measures $\nu_0$ are $\nu_1$ are mutually singular but ``intermingled'', as they have the same support.

\begin{proof}
Let $\nu$ be the unique $T$-invariant Borel probability measure.
Since $\nu$ is $T^2$-invariant but $T^2$ is not uniquely ergodic, there exists an $T^2$-invariant measure $\nu_0$ different from $\nu$. We can assume that $\nu_0$ is ergodic with respect to $T^2$. Let $\nu_1 \coloneqq T_*(\nu_0)$. Then $T_*(\nu_1) = T_*^2 (\nu_0) = \nu_0$. The probability measure $\frac{1}{2}\nu_0+ \frac{1}{2}\nu_1$ is $T$-invariant and therefore it must coincide with $\nu$. Like $\nu_0$, the measure $\nu_1$ is ergodic with respect to $T^2$. Since $\nu_0 \neq \nu$, we have $\nu_1 \neq \nu_0$. So we found two different ergodic measures for $T^2$. Note that the ergodic decomposition of $\nu$ with respect to $T^2$ is $\frac{1}{2}\nu_0+ \frac{1}{2}\nu_1$. Now, suppose $\tilde{\nu}_0$ is another ergodic measure for $T^2$. Then, letting $\tilde{\nu}_1 \coloneqq T_*(\nu_1)$, a repetition of the argument above yields that $\frac{1}{2}\tilde \nu_0+ \frac{1}{2}\tilde \nu_1$ is an ergodic decomposition of $\nu$ with respect to $T^2$. By uniqueness of ergodic decomposition, we conclude that $\tilde{\nu}_0$ equals either $\nu_0$ or $\nu_1$. So $T^2$ admits exactly two ergodic measures, as we wanted to show. Since $T^2$ is minimal, each of those measures has full support.
\end{proof}

\subsection{Walters cocycles}\label{ss.Walters}

We now describe the class of nonuniformly hyperbolic cocycles introduced by Walters \cite{Walters}. 

\begin{theorem}[Walters \cite{Walters}] \label{t.Walters}
Let $T \colon Y \to Y$ be a Veech-like map.
Then, for every $\lambda_0>0$, there exists a continuous map $B \colon Y \to \mathrm{SL}^{\smallpm}(2,\R)$ 
such that the cocycle $(T,B)$ is not uniformly hyperbolic and its Lyapunov exponent $\lambda(T,B,\nu)$ with respect to the unique $T$-invariant probability measure $\nu$ equals $\lambda_0$. Furthermore, the matrices of the cocycle are of the form
\begin{equation}\label{e.Walters}
B(y) = \begin{pmatrix} 0 & b(y) \\ 1/b(y) & 0 \end{pmatrix} \, .
\end{equation}
\end{theorem}

Any cocycle as in the conclusion of the theorem will be called a \emph{Walters cocycle}.
For the convenience of the reader, we provide a proof of \cref{t.Walters}.
We note that the cocycle actually defined by Walters \cite{Walters} is slightly different from ours (the determinant was not normalized to $-1$), and his proof used a different argument (see \cref{r.liminf_zero} below). 

\begin{proof}[Proof of \cref{t.Walters}]
Let $\nu_0$ and $\nu_1$ be the two $T^2$-invariant probability measures provided by \cref{l.double_ergodicity}.
Since $\nu_0 \neq \nu_1$,
given a number $\lambda_0>0$, we can find a continuous function $\varphi$ such that 
\begin{equation}\label{e.lambda}
\int_Y \varphi \, d\nu_0 - \int_Y \varphi \, d\nu_1  = 2\lambda_0 \, .
\end{equation}
Let $b \coloneqq e^\varphi$ and define $B$ by formula \eqref{e.Walters}.
We claim that: 
\begin{itemize}
\item the Lyapunov exponent $\lambda(T,B,\nu)$ equals $\lambda_0$;
\item for $\nu_0$-a.e.\ $y$, the Oseledets spaces are $E^\mathrm{u} (y) = \R \left( \begin{smallmatrix} 1 \\ 0 \end{smallmatrix} \right)$ and $E^\mathrm{s}(y) = \R \left( \begin{smallmatrix} 0 \\ 1 \end{smallmatrix} \right)$; 
\item for $\nu_1$-a.e.\ $y$, the Oseledets spaces are $E^\mathrm{u} (y) = \R \left( \begin{smallmatrix} 0 \\ 1 \end{smallmatrix} \right)$ and $E^\mathrm{s}(y) = \R \left( \begin{smallmatrix} 1 \\ 0 \end{smallmatrix} \right)$.
\end{itemize}
Consider the function 
\begin{equation}\label{e.psi}
\psi \coloneqq \varphi  - \varphi \circ T \, .
\end{equation}
By \cref{l.double_ergodicity}, $T_* \nu_0 = \nu_1$, and so 
\begin{equation}\label{e.two_lambda}
\int \psi \, d\nu_0 = \int \varphi \, d\nu_0 -   \int \varphi \circ T \, d\nu_0 = \int \varphi \, d\nu_0 -   \int \varphi  \, d\nu_1 = 2\lambda_0 \, .
\end{equation}
Similarly, $\int \psi \, d\nu_1 = -2 \lambda_0$.
The relevance of the function $\psi$ comes from the fact that
\begin{equation}
B^{(2)}(y) = B(T(y)) B(y) = \begin{pmatrix} e^{\psi(y)} & 0 \\ 0 & e^{-\psi(y)} \end{pmatrix} \, ;
\end{equation}
more generally, $B^{(2n)}(y) $ is a diagonal matrix of determinant $1$ whose top-left entry is
\begin{equation}
\exp \sum_{j = 0}^{n-1} \psi(T^{2j}(y)) \, .
\end{equation}
Note that the log of this quantity is a Birkhoff sum with respect to $T^2$.
So, by the ergodic theorem, for $\nu_0$-almost every $y$, the Lyapunov exponent of the vector $\left( \begin{smallmatrix} 1 \\ 0 \end{smallmatrix} \right)$ is $\frac{1}{2} \int \psi \, d\nu_0$, which by \eqref{e.two_lambda}, equals $\lambda_0$. Thus the expanding Oseledets space is $E^\mathrm{u} (y) = \R \left( \begin{smallmatrix} 1 \\ 0 \end{smallmatrix} \right)$, as claimed above. The other claims are proved similarly.

Since both measures $\nu_0$ and $\nu_1$ have full support, the Oseledets splitting cannot coincide $\nu$-a.e.\ with a continuous splitting. 
So the cocycle $(T,B)$ is not uniformly hyperbolic. 
\end{proof}

\begin{remark}\label{r.liminf_zero} 
Walters \cite{Walters} showed that the cocycle is not uniformly hyperbolic by proving that there exist points $y \in Y$ for which the sequence $\frac{1}{n} \log \|B^{(n)}(y)\|$ does not converge.
In fact, the set of accumulation points of the sequence above is $[0,\lambda_0]$ for all points $y$ in a residual subset of $Y$; see \cite[Corollary~4.13]{BPS}.
\end{remark}

\subsection{Monochromatic extensions}

The key technical result of this paper is the following, whose proof is postponed to \cref{s.main_positive}:

\begin{theorem}\label{t.main_positive}
Let $T \colon X \to X$ be a biLipschitz homeomorphism of a compact metric space.
Let $Y \subseteq X$ be a nonempty $T$-invariant compact set such that the restriction $T|_Y$ is uniquely ergodic with a corresponding measure $\nu$.
Let $B \colon X \to \mathrm{SL}^{\smallpm}(2,\R)$ be a continuous map such that $B(x)$ is a nonnegative matrix for all $x \in X$.
Assume that the Lyapunov exponent $\lambda_0 \coloneqq \lambda(T|_Y,B|_Y,\nu)$ is positive.
Then there exists a continuous map $A \colon X \to \mathrm{SL}^{\smallpm}(2,\R)$ such that $A(y)=B(y)$ for all $y \in Y$ and 
$\lambda(T,A,\mu) = \lambda_0$ for every ergodic Borel probability measure $\mu$. 
\end{theorem}

Clearly, if the cocycle $(T|_Y,B|_Y)$ is not uniformly hyperbolic, then the application of \cref{t.main_positive} produces a monochromatic nonuniformly hyperbolic cocycle $(T,A)$.

\begin{remark}
We do not know if assuming that the cocycle $(T|_Y,B|_Y)$ is monochromatic with positive exponent is  it is sufficient for the validity of \cref{t.main_positive}; our proof does use unique ergodicity.
On the other hand, the biLipschitz hypothesis on $T$ is probably unnecessary, but we found it convenient for the construction.
\end{remark}

\subsection{Proof of the main results}

We need one final ingredient:

\begin{theorem}[Krieger's embedding theorem {\cite[Corollary 10.1.9]{LindMarcus}}] \label{t.embedding}
Let $Y$ be a shift space and let $X$ be a mixing SFT.
Then $Y$ can be embedded as a proper subshift of $X$ if and only if the following two conditions are satisfied:
\begin{enumerate}
	\item\label{i.entropy} $h_\mathrm{top}(Y) < h_\mathrm{top}(X)$ and 
	\item\label{i.periodic} for each $n\ge 1$, the number $\mathfrak{P}_Y(n)$ of periodic orbits with minimal period $n$ for the shift $Y$ is less than or equal to the corresponding number $\mathfrak{P}_X(n)$. 
\end{enumerate}
\end{theorem}

We now combine all previous theorems 
to prove the existence of nonuniformly hyperbolic monochromatic cocycles. 
We begin with the following variation of \cref{t.main_SFT} which allows for dynamics beyond SFTs:

\begin{theorem}\label{t.main_embed}
Let $T \colon X \to X$ be a biLipschitz homeomorphism of a compact metric space.
Suppose that $T$ has a compact invariant set $\Lambda$ containing more than one point and such that the restriction $T|_\Lambda$ is topologically conjugate to a mixing subshift of finite type.
Fix a constant $\lambda_0>0$.
Then there exists a continuous map $A \colon X \to \mathrm{SL}^{\smallpm}(2,\R)$ such that
\begin{itemize}
\item $\lambda(T,A,\mu) = \lambda_0$ for every ergodic Borel probability measure $\mu$ and
\item the cocycle $(T,A)$ is not uniformly hyperbolic.
\end{itemize}
\end{theorem}

For example, the base dynamics $T$ can be any $C^1$ diffeomorphism of a compact manifold that admits a hyperbolic fixed point with a transverse homoclinic point: use \cite[Theorem~6.5.5]{KatokH}. 

\begin{proof}[{Proof of \cref{t.main_embed}}]
We are given a number $\lambda_0>0$ and a biLipschitz homeomorphism $T \colon X \to X$ admitting an  compact invariant set $\Lambda$ such that $T|_\Lambda$ is topologically conjugate to a mixing SFT. Furthermore, we assume that $\Lambda$ is not a singleton, which implies that $T|_\Lambda$ has positive topological entropy. 
By \cref{t.existence_Veech}, there exists a Veech-like shift map with zero topological entropy. 
Note that Veech-like maps have no periodic points. 
Therefore, by \cref{t.embedding}, we can find a compact invariant set $Y \subset \Lambda$ such that $T|_Y$ is a Veech-like map. 
Use \cref{t.Walters} to construct a Walters cocycle $B \colon Y \to \mathrm{SL}^{\smallpm}(2,\R)$ with Lyapunov exponent $\lambda_0$.
Recall that $B(y)$ is the nonnegative antidiagonal matrix \eqref{e.Walters} where $b \colon Y \to \R_\smallplus$ is a continuous function.
Tietze extension theorem provides a continuous extension of $b$ to the whole space $X$, which we denote by the same letter $b$.
Then formula \eqref{e.Walters} defines a map $B \colon X \to \mathrm{SL}^{\smallpm}(2,\R)$ that takes nonnegative values. 
Now we can apply \cref{t.main_positive} and find a continuous $A \colon X \to \mathrm{SL}^{\smallpm}(2,\R)$ such that $A|_Y = B|_Y$ and the cocycle $(T,A)$ is monochromatic, i.e.\ $\lambda(T,A,\mu) = \lambda_0$ for every ergodic Borel probability measure $\mu$. The cocycle $(T,A)$ is not uniformly hyperbolic, since its restriction to $Y$ is a Walters cocycle.
\end{proof}

Next, we prove \cref{t.main_SFT}, stated in the Introduction. 

\begin{proof}[{Proof of \cref{t.main_SFT}}]
Now we are given $\lambda_0>0$ and a SFT $T \colon X \to X$ where $X$ is an infinite set.
Consider a spectral decomposition, that is, a decomposition of $X$ into disjoint nonempty compact subsets
\begin{equation}
X = \bigsqcup_{i=1}^m X_i \, , \quad
X_i = \bigsqcup_{j=1}^{n_i} X_{i,j} \quad 
\text{such that} \quad  T(X_{i,j}) = X_{i,j+1 \bmod n_i}
\end{equation}
and each $T^{n_i}|_{X_{i,1}}$ is mixing.
Reindexing if necessary, we can assume that $X_{1,1}$ is not a singleton. 

Just like in the previous proof, we can find a compact invariant set $Y \subset X_{1,1}$ such that the restricted map $T^{n_1}|_Y$ is Veech-like, say with invariant measure $\nu$.
Then we can find a Walters cocycle $(T^{n_1}|_Y,B)$ whose Lyapunov exponent with respect to $\nu$ equals $n_1 \lambda_0$.
Let $\hat{B} \colon X \to \mathrm{SL}^\smallpm(2,\R)$ be a continuous extension of the map $B$ such that $\hat{B}(x)$ is a nonpositive antidiagonal matrix if $x \in X_{1,1}$ and is the identity matrix otherwise. 
Note that the restriction of $T$ to the invariant set $\hat{Y} \coloneqq \bigcup_{j=0}^{n_1-1} T^j(Y)$ is uniquely ergodic with corresponding measure $\hat{\nu} \coloneqq \frac{1}{n_1}\sum_{j=0}^{n_1-1} T^j_* \nu$. 
Furthermore, $\lambda(T,\hat{B},\hat \nu) = \lambda_0$. 
The cocycle $(T|_{\hat{Y}}, \hat{B}|_{\hat{Y}})$ is not uniformly hyperbolic, because $(T^{n_1}|_Y,B)$ is not. 
Applying \cref{t.main_positive}, we find a monochromatic cocycle $(T,A)$, where $A$ coincides with $\hat{B}$ on the set $\hat{Y}$. In particular, the cocycle $(T,A)$ cannot be uniformly hyperbolic.
\end{proof}

Our main results are proved, modulo \cref{t.main_positive,t.existence_Veech}, whose proofs are presented in the next two (entirely independent) sections.

\section{Perturbing nonnegative cocycles}\label{s.main_positive}

\subsection{Summary}\label{ss.positive_summary}

In this section, we will prove \cref{t.main_positive}.
Let us provide some highlights of the proof. 
We define an auxiliary cocycle map $P$ (see formula \eqref{e.P}) which coincides with $B$ on the set $Y$, and is formed by strictly positive matrices on its complement $X \setminus Y$. We call $P$ the \emph{positive perturbation} of $B$.
The matrix $P$ has a very rough modulus of regularity near $Y$: in fact, the off-diagonal entries of the matrix grow very quickly away from $Y$. 

Since the restriction of the cocycle $B$ to the invariant set $Y$ may (and in the interesting cases will) fail to be uniformly hyperbolic, it is possible that the cocycle products along $Y$ display \emph{cancellations} of the form $\|B^{(n+m)}(y)\| \ll \|B^{(m)}(T^n y)\| \, \|B^{(n)}(y)\|$. On the other hand, away from $Y$ our matrices $P$ are strictly positive. As a general principle, positivity prevents cancellations: see e.g.\ \cite[Lemma~2.2]{FanWu} for a simple materialization of this idea. We follow this principle, but the actual implementation is rather delicate. We initially use unique ergodicity of $T|_Y$ to extract some uniform information about the cocycle $B|_Y$: we identify \emph{time scales} where the norms of the products of $B$ matrices behave as expected (see \cref{l.time_scale}). This information is then combined with the quantitative positivity of the matrices $P$ in order to control finite-time expansion rates near $Y$, in a precise formulation that takes into account closeness to $Y$ (see \cref{l.key_estimate}). 
 
Using positivity of the matrices $P$ for a second time, we construct a continuous invariant field of lines on the invariant set $X \setminus Y$ (see \cref{l.bundle}). While the construction is reminiscent of the Perron-Frobenius theorem, the fact that the domain $X \setminus Y$ is noncompact introduces complications. Here we use, once again, the rapid growth of the off-diagonal entries of the matrices $P$ away from $Y$.

At this point, we are reduced to a one-dimensional problem: we study the logarithm of the expansion rate along the invariant line field. If this function were continuously cohomologous to the constant $\lambda_0$ (i.e.\ the Lyapunov exponent), the construction would be complete. What we actually do is solving an approximate cohomological equation for this expansion rate, defined on the noncompact set $X \setminus Y$. This equation, despite degenerating as we approach $Y$, has a vanishing error term (see \cref{l.approx_cob}). Once in possession of the solution, it is simple to define a last perturbation of the cocycle displaying the monochromatic property (see equation~\eqref{e.final_perturb}).

\medskip
\noindent\textbf{Conventions for this section:} 
We will use $(\epsilon_k)_{k \ge 0}$ to denote any sequence of positive numbers that converges to $0$; 
the specific value of the sequence may change with each occurrence.
Similarly, $C$ will denote any positive constant.
That is, $\epsilon_k$ and $C$ are simply positive versions of the symbols $o_k(1)$ and $O(1)$, respectively.

\subsection{Fixing the time scales} 

For the remainder of this section, let $X$, $T$, $Y$, $\nu$, $B$, and $\lambda_0$ be as in the statement of \cref{t.main_positive}. 

\begin{lemma} 
We have the following limits:
\begin{alignat}{2}
\label{e.SAET}
\lambda_0 &= 
\lim_{n \to \infty} \frac{1}{n}  \int_Y &&\log \| B^{(n)}(y) \| \, d\nu(y)  \, , \\
\label{e.SUSAET}
\lambda_0 &= 
\lim_{n \to \infty} \frac{1}{n}  \sup_{y \in Y}  &&\log \| B^{(n)}(y) \| \, .
\end{alignat}
\end{lemma}

\begin{proof}
The first limit comes follows the subadditive ergodic theorem \cite[Theorem~3.4.2]{DamanikF}, while the second one follows from the semi-uniform subadditive ergodic theorem \cite[Theorem~A.3]{Morris}.
\end{proof}

The next step is to fix a superexponential sequence $(m_k)$ of \emph{time scales} along which we have some specific control of norms.

\begin{lemma}[Time scales]\label{l.time_scale}
There exists a sequence of positive integers $(m_k)_{k \ge 1}$
such that 
\begin{equation}\label{e.superexp}
	\frac{m_{k+1}}{m_k} \nearrow \infty \, ,
\end{equation}
and for all $y \in Y$, $k \ge 1$, and $n \ge m_{k+1}$, there exist integers 
$q$, $r$, $s$ such that 
\begin{gather}
\label{e.silly_decomposition}
n = r + q m_k + s \, ,
\quad q>0  \, ,
\quad |r|<m_k \, , 
\quad |s|<m_k \, , \quad \text{and} \\
\label{e.mk_property}
\frac{1}{q m_k} \sum_{j=0}^{q-1} \log \| B^{(m_k)}(T^{r + j m_k} (y)) \| > \lambda_0 - \epsilon_k  \, .
\end{gather}
\end{lemma}

\begin{proof} 
The sequence $(m_k)$ will be constructed recursively. 
It will have the following extra property:
\begin{equation}\label{e.extra}
\int_Y \frac{1}{m_k} \log \|B^{(m_k)}\| \, d\nu > \lambda_0 - \frac{1}{k} \, .
\end{equation}

Using \eqref{e.SAET}, we can choose $m_1$ such that \eqref{e.extra} holds with $k=1$.
Now fix $k$ and as an induction hypothesis suppose that a number $m_k$ with property \eqref{e.extra} is already chosen.
Since $T|_Y$ is uniquely ergodic, the Birkhoff averages of the continuous function $\frac{1}{m_k} \log \|B^{(m_k)}\|$ converge uniformly to its integral with respect to the unique invariant probability $\nu$. Therefore, we can find $m_{k+1}$ such that for all $N \ge \frac{m_{k+1}}{2}$ and $y \in Y$, we have 
\begin{equation}\label{e.morning}
\frac{1}{N}\sum_{i=0}^{N-1} \frac{1}{m_k} \log \| B^{(m_k)} (T^i y) \| > \lambda_0 - \frac{2}{k} \, .
\end{equation}
Increasing $m_{k+1}$ if necessary, we assume that $m_{k+1} > (k+1)m_k$ and furthermore property  \eqref{e.extra} holds with $k+1$ in the place of $k$ (which is feasible because of \eqref{e.SAET}). 
This completes the inductive defintion of the sequence $(m_k)$.

Now let us check that the sequence has the desired properties.
Fix $k \ge 1$,  $y \in Y$, and $n \ge m_{k+1}$. Let $q \coloneqq \left\lfloor \frac{n}{m_k} \right\rfloor$ and $N \coloneqq q m_k$. Note that $N > n - m_k \ge \frac{m_{k+1}}{2}$, so  inequality \eqref{e.morning} holds. 
We rewrite the left hand side of this inequality as follows:
\begin{equation}
\frac{1}{N}\sum_{i=0}^{N-1} \frac{1}{m_k} \log \| B^{(m_k)} (T^i y) \|
= \frac{1}{m_k} \sum_{r=0}^{m_k-1} 
\underbrace{\frac{1}{qm_k} \sum_{j=0}^{q-1} \log \| B^{(m_k)}(T^{r + j m_k} (y)) \|}_{\eqcolon a_r}
\, .
\end{equation}
Choose and fix some $r \in \ldbrack 0, m_k - 1 \rdbrack$ such that $a_r > \lambda_0 - \frac{2}{k}$. 
That is, property~\eqref{e.mk_property} holds where $\epsilon_k$ is actually $\frac{2}{k}$. Finally, let $s \coloneqq (n - N) - r$. Since both $n - N$ and $r$ belong  to $\ldbrack 0, m_k - 1 \rdbrack$, we have $|s|<m_k$. Thus relations \eqref{e.silly_decomposition} hold.
\end{proof}

\subsection{Fixing the space scales}

Rescaling the metric if necessary, we assume that
\begin{equation}\label{e.one_tenth}
\diam X < \tfrac{1}{10} \, .
\end{equation}

The next step is to fix a sequence $(\delta_k)_{k \ge 1}$ of \emph{space scales} associated to the sequence of time scales $(m_k)$.
Essentially, we need that $\delta_k \searrow 0$, and that this convergence is extremely fast. 
More precisely, we impose the following conditions:  
\begin{align}
\delta_1 &\coloneqq \diam X \, , \\ 
\label{e.need_for_anchors}
\delta_k &< e^{-e^{m_{k-4}}} \text{ if } k>4 \, , \\ 
\label{e.page26}
\delta_{k+1} &< \delta_k e^{-k \, m_{k+1}} \, .
\end{align}

Consider the sequence of \emph{Bowen metrics} defined as follows:
\begin{equation}
\mathsf{d}_n (x_1,x_2) \coloneqq \max_{-n \le i \le n}  \mathsf{d} (T^i(x_1),T^i(x_2)) \, .
\end{equation}
By uniform continuity of $T^{\pm 1}$, each Bowen metric $\mathsf{d}_n$ is uniformly comparable to the initial metric $\mathsf{d}$. Therefore some additional care in the choice of space scales ensures the following property:
\begin{equation}
\label{e.delta_Bowen}
\text{for all $k>1$ and $x_1, x_2 \in X$, if $\mathsf{d}(x_1,x_2) \le \delta_{k}$, then $\mathsf{d}_{m_{k+1}}(x_1,x_2) < \delta_{k-1}$.}
\end{equation}

The last imposition on the sequence $(\delta_k)$ is the following property:
\begin{multline}
\label{e.twice_norm}
\text{for all $k>0$ and $x_1, x_2 \in X$, if $\mathsf{d}(x_1,x_2) \le \delta_{k}$, then}\\
\text{$\|B^{(n)}(x_1)\| < 2 \|B^{(n)}(x_2)\|$ for all $n \in \ldbrack 0 , m_k \rdbrack$,}
\end{multline}

which can be enforced simply because the matrix map $B$ is uniformly continuous.

\subsection{Positive perturbation}

In the next lemma, we fix an auxiliary continuous function $\phi$ with a very rough
modulus of continuity. 

\begin{lemma}\label{l.phi}
There exists an increasing homeomorphism $\phi \colon [0,\frac{1}{10}] \to [0,\frac{1}{2}]$ such that
\begin{alignat}{2}
\label{e.phi_anchors}
\phi(\delta_k) &= e^{-m_{k-4}} &\quad &\text{for each } k > 4 \quad \text{and}
\\
\label{e.phi_explosion}
\phi(t) &> \frac{1}{- \log t} &\quad &\text{for every } t \in (0, \tfrac{1}{10}] \, .
\end{alignat}
\end{lemma}

\begin{proof}
Let $\psi$ be the homeomorphism from $[0,\tfrac{1}{\log 10}]$ to $[0, \tfrac{1}{2}]$ 
defined as follows.
We fix the following \emph{anchor points}:
\begin{equation}
\psi(0) \coloneqq 0 \, , \quad 
\psi\big(\tfrac{1}{\log 10} \big) \coloneqq \tfrac{1}{2} \, , \quad 
\psi\big(\tfrac{1}{-\log \delta_k}\big) \coloneqq e^{-m_{k-4}} \text{ for all } k>4  \, ;
\end{equation} 
then we interpolate linearly between any two consecutive anchor points. 
It follows from \eqref{e.need_for_anchors} that $\psi(s)>s$ for all $s \in (0,\tfrac{1}{\log 10}]$. Then the function $\phi(t) \coloneqq \psi \big(\tfrac{1}{-\log t}\big)$ has all the desired properties.
\end{proof}

Define the following continuous functions of $x \in X$:
\begin{align}
\label{e.theta}
\theta(x) &\coloneqq \phi(\mathsf{d}(x,Y)) \, , \\ 
\label{e.H}
H(x) &\coloneqq 
\begin{pmatrix}
\cosh \theta(x) & \sinh \theta(x) \\ 
\sinh \theta(x) & \cosh \theta(x) 
\end{pmatrix} \, , \quad \text{and} \\
\label{e.P} 
P(x) &\coloneqq B(x) H(x) \, .
\end{align}
If $x \in Y$, then $H(x) = \mathrm{Id}$ and $P(x) = B(x)$, while if $x \not\in Y$, then the matrices $H(x)$ and $P(x)$ are (entrywise) positive.
In both cases, $\det H(x) = 1$ and $\det P(x) = \pm 1$.

The cocycle $(T,P)$ is called the \emph{positive perturbation} of the cocycle $(T,B)$. (In reality, it is a small perturbation only near $Y$; away from $Y$ the perturbation is rather large). 

\medskip

The following simple bound relies fundamentally on positivity and will be essential in the proof of \cref{l.key_estimate} below.

\begin{lemma}\label{l.positivity}
For all $x \in X$ and $n \ge 0$, if $v$ is a nonnegative vector, then
\begin{equation}\label{e.positivity}
\|P^{(n)}(x) v \| \ge \theta(x) \| B^{(n)}(x) \| \, \| v\| \, .
\end{equation}
\end{lemma}

\begin{proof}
Since $\cosh t \ge \sinh t \ge t$ for every $t \ge 0$, we have 
\begin{equation}
H(x) \ge \theta(x) U \, ,  \quad \text{where} \quad 
U \coloneqq \begin{pmatrix} 1 & 1 \\ 1 & 1 \end{pmatrix} 
\end{equation}
and we are using the entrywise partial order on the set of matrices.
On the other hand, since $\cosh t \ge 1$, we have $H(T^i x) \ge \mathrm{Id}$ for every $i$.
Since product of nonnegative matrices is increasing with respect to the partial order, it follows that
\begin{align}
P^{(n)}(x) 
&= B(T^{n-1} x) H(T^{n-1} x) \cdots B(T x) H(T x) B(x) H(x) \\ 
&\ge \theta(x) B(T^{n-1} x) \cdots B(T x) B(x) U \\ 
&= \theta(x) B^{(n)}(x) U \, .
\end{align}
Next, note that
\begin{equation}
U v \ge \|v\| u \, ,  \quad \text{where} \quad 
u \coloneqq \begin{pmatrix} 1 \\ 1 \end{pmatrix} \, .
\end{equation}
Let $w$ be a nonnegative unit (column-)vector such that $\|B^{(n)}(x) w\| = \|B^{(n)}(x)\|$. 
Note that $w \le u$. 
Combining those observations,
\begin{align}
P^{(n)}(x) 
&\ge \theta(x) B^{(n)}(x) U v\\	
&\ge \theta(x) \|v\| \, B^{(n)}(x) u\\	
&\ge \theta(x) \|v\| \, B^{(n)}(x) w \, .
\end{align}
Taking norms, we obtain the desired inequality.
\end{proof}

\subsection{Avoiding cancellations}

The next step establishes the following lemma which
as explained informally in \cref{ss.positive_summary}, uses positivity in a fundamental way and is a central component of our proof:

\begin{lemma}\label{l.key_estimate}
For all $k \in \N$, $x \in X$, $n \in \Z_\smallplus$, and $v$ a nonnegative unit vector,
if 
\begin{equation}\label{e.two_ranges}
\delta_{k+1} \le \mathsf{d}(x,Y) \le \delta_{k} \quad \text{and} \quad
m_{k} \le n \le m_{k+1} \, ,
\end{equation}
then 
\begin{equation}\label{e.key_estimate}
\left| \frac{1}{n} \log \big\| P^{(n)}(x) v \big\| - \lambda_0 \right| < \epsilon_k \, .
\end{equation}
\end{lemma}

Recall that, as explained at the end of \cref{ss.positive_summary}, $(\epsilon_k)$ denotes any positive sequence converging to $0$, and its specific value may (and will) change with each occurrence. 

\begin{proof}[Proof of \cref{l.key_estimate}]
We can assume that $k>2$.
Fix a point $x\in X$ such that $\delta_{k+1} \le \mathsf{d}(x,Y) \le \delta_{k}$.
Let $y \in Y$ be such that $\mathsf{d}(x,y) = \mathsf{d}(x,Y)$.
We begin the proof by establishing two simple but crucial inequalities, illustrated in \cref{f.far_close}.
Since $\mathsf{d}(x,y) \le \delta_k$, property~\eqref{e.delta_Bowen} implies that
\begin{equation}\label{e.dont_go_too_far}
\mathsf{d}_{m_{k+1}} (x,y) \le \delta_{k-1} \, .
\end{equation}
On the other hand, 
we claim that for all integers $i$ in the range $|i| \le m_{k+1}$,
\begin{equation}\label{e.dont_get_too_close}
\mathsf{d}(T^i x, Y) \ge \delta_{k+2} \, .
\end{equation}
Indeed, if the inequality fails, then there exists $y' \in Y$ such that $\mathsf{d}(T^i x, y') < \delta_{k+2}$. In that case, property~\eqref{e.delta_Bowen} gives $\mathsf{d}_{m_{k+3}}(T^i x, y') < \delta_{k+1}$. In particular, since $|i| \le m_{k+1} < m_{k+3}$, we have $\mathsf{d}(x,T^{-i} y') < \delta_{k+1}$, contradicting the fact that $\mathsf{d}(x,Y) \ge \delta_{k+1}$.

\begin{figure}
\begin{tikzpicture}[scale=1.5]
\filldraw[fill=red!20,draw=red!50!black] plot[smooth cycle] coordinates {(-2.7,0) (-2, 1) (2,1) (2.7,0) (2,-1) (-2,-1)};
\draw[red] (-1.5,-.25) node{$Y$};

\node[red] (y0) at (-2.0,1.00) {\tiny $\bullet$};
\node[red] (y1) at (-1.6,0.96) {\tiny $\bullet$};
\node[red] (y2) at (-1.2,0.90) {\tiny $\bullet$};
\node[red] (y3) at (-0.8,0.82) {\tiny $\bullet$};
\node[red] (y4) at (-0.4,0.72) {\tiny $\bullet$};
\node[red] (y5) at ( 0.0,0.60) {\tiny $\bullet$};
\node[red] (y6) at ( 0.4,0.46) {\tiny $\bullet$};
\node[red] (y7) at ( 0.8,0.30) {\tiny $\bullet$};
\node[red] (y8) at ( 1.2,0.12) {\tiny $\bullet$};
\node[red] (y9) at ( 1.6,-0.08) {\tiny $\bullet$};
\node[red](y10) at ( 2.0,-0.30) {\tiny $\bullet$};
\node(foot) 	at ( 0.0,1.1) {};
\node[blue] (x0) at (-2.0,2.20) {\tiny $\bullet$};
\node[blue] (x1) at (-1.6,2.04) {\tiny $\bullet$};
\node[blue] (x2) at (-1.2,1.92) {\tiny $\bullet$};
\node[blue] (x3) at (-0.8,1.84) {\tiny $\bullet$};
\node[blue] (x4) at (-0.4,1.80) {\tiny $\bullet$};
\node[blue] (x5) at ( 0.0,1.80) {\tiny $\bullet$};
\node[blue] (x6) at ( 0.4,1.84) {\tiny $\bullet$};
\node[blue] (x7) at ( 0.8,1.92) {\tiny $\bullet$};
\node[blue] (x8) at ( 1.2,2.04) {\tiny $\bullet$};
\node[blue] (x9) at ( 1.6,2.20) {\tiny $\bullet$};
\node[blue](x10) at ( 2.0,2.40) {\tiny $\bullet$};

\tikzstyle{flechav}=[-{Straight Barb[length=0.5mm]},thin,red];
\draw[flechav] (y0) node[below] {$y$} to (y1);
\draw[flechav] (y1) to (y2);
\draw[flechav] (y2) to (y3);
\draw[flechav] (y3) to (y4);
\draw[flechav] (y4) to (y5);
\draw[flechav] (y5) to (y6);
\draw[flechav] (y6) to (y7);
\draw[flechav] (y7) to (y8);
\draw[flechav] (y8) to (y9);
\draw[flechav] (y9) to (y10) node[below] {$T^{m_{k+1}} y$};

\tikzstyle{flechaa}=[-{Straight Barb[length=0.5mm]},thin,blue];
\draw[flechaa] (x0) node[left] {$x$} to (x1);
\draw[flechaa] (x1) to (x2);
\draw[flechaa] (x2) to (x3);
\draw[flechaa] (x3) to (x4);
\draw[flechaa] (x4) to (x5);
\draw[flechaa] (x5) to (x6);
\draw[flechaa] (x6) to (x7);
\draw[flechaa] (x7) to (x8);
\draw[flechaa] (x8) to (x9);
\draw[flechaa] (x9) to (x10) node[right] {$T^{m_k+1} x$};

\tikzstyle{flechadist}=[shorten >=-3pt, shorten <=-3pt, {Stealth}-{Stealth}];
\draw[flechadist] (x0) -- (y0)	node[midway,left]		{$\delta_{k+1} \le $} 
								node[midway,right]		{$\le \delta_{k}$};
\draw[flechadist] (x5) -- (foot)node[midway,left]		{$\delta_{k+2} \le $};
\draw[flechadist] (x9) -- (y9)	node[near start,right]	{$\le \delta_{k-1}$};
\end{tikzpicture}
\caption{Inequality \eqref{e.dont_get_too_close} prevents the points $x,Tx,\dots,T^{m_{k+1}} x$ from getting too close to $Y$, while 
inequality \eqref{e.dont_go_too_far} guarantees that they never go too far away from the respective points $y,Ty,\dots,T^{m_{k+1}} y$.}\label{f.far_close}
\end{figure}

Applying \cref{l.time_scale}, we obtain a decomposition $n = r + q m_{k-1} + s$ where $q>0$, $|r|<m_{k-1}$, $|s|<m_{k-1}$, and 
and 
\begin{equation}\label{e.comeco}
\frac{1}{q m_{k-1}} \sum_{j=0}^{q-1} \log \| B^{(m_{k-1})}(T^{r + j m_{k-1}} (y)) \| > \lambda_0 - \epsilon_k  \, .
\end{equation}

Consider points $x_j \coloneqq T^{r + j m_{k-1}} (x)$. 
We factorize
\begin{equation}
P^{(n)}(x) = R_2 P_{q-1} \cdots P_1 P_0 R_1 \, ,
\end{equation}
where $R_1 \coloneqq P^{(r)}(x)$, $P_j \coloneqq P^{(m_{k-1})}(x_j)$, and 
$R_2 \coloneqq P^{(s)}(x_q)$.

Let $\theta_0 \coloneqq \phi(\delta_{k+2})$, where the function $\phi$ comes from \cref{l.phi}. 
It follows from \eqref{e.dont_get_too_close} that $\theta(x_j) \ge \theta_0$.
Therefore, by \cref{l.positivity}, for all nonnegative vectors $w$, we have
\begin{equation}
\| P_j w\| \ge \theta_0 \|B_j\| \|w\| \, , \quad \text{where} \quad B_j \coloneqq B^{(m_{k-1})}(x_j) \, .
\end{equation}
Applying this property recursively, we obtain
\begin{equation}\label{e.puzzle1}
\| P_{q-1} \cdots P_0 w\| \ge \theta_0^q \|B_{q-1}\| \cdots \|B_0\| \|w\| \, .
\end{equation}
By \eqref{e.phi_anchors}, we actually have $\theta_0 = e^{-m_{k-2}}$, so 
\begin{equation}\label{e.puzzle2}
	\theta_0^q  = e^{-q m_{k-2}} > e^{-\epsilon_k q m_{k-1}} > e^{-\epsilon_k n} \, ,
\end{equation}

Next, consider the points $y_j \coloneqq T^{r + j m_{k-1}} (y)$ and the corresponding matrix products $\tilde{B}_j \coloneqq B^{(m_{k-1})}(y_j)$.
By \eqref{e.dont_go_too_far} and \eqref{e.twice_norm}, 
we have 
\begin{equation}\label{e.puzzle3}
\|B_j\| > \tfrac{1}{2} \| \tilde{B}_j\| 
\quad \text{for each } j \in \ldbrack 0, q-1 \rdbrack \, .
\end{equation}
On the other hand, by \eqref{e.comeco},
$\|\tilde{B}_{q-1}\| \cdots \|\tilde{B}_0\| > e^{(\lambda_0 - \epsilon_k)qm_{k-1}}$.
Observe  that $|q m_{k-1} - n| < 2m_{k-1} = \epsilon_k n$.
We combine the bounds above: given a nonnegative unit vector $v$, letting $w \coloneqq \frac{R_1 v}{\|R_1 v\|}$, we have
\begin{align}
\|P^{(n)}(x) v \| 
&\ge \|R_2^{-1} \|^{-1} \,  \|P_{q-1} \cdots P_0 w\|  \, \|R_1^{-1} \|^{-1} \\
&> e^{-Cm_{k-1}} \, \theta_0^q \,  \|B_{q-1}\| \cdots \|B_0\| \\ 
&> e^{-Cm_{k-1}} \, (\theta_0/2)^q \, \|\tilde{B}_{q-1}\| \cdots \|\tilde{B}_0\|  \\
&> e^{(\lambda_0 - \epsilon_k)n} \, .
\end{align}
That is, $\frac{1}{n} \log \|P^{(n)}(x) v \|$ admits the lower bound $\lambda_0 - \epsilon_k$, which confirms part of the desired inequality \eqref{e.key_estimate}.
Obtaining an upper bound uses different (and simpler) arguments, so we will prove it separately in \cref{l.trouxa}.
\end{proof}

\begin{lemma}\label{l.trouxa}
If $x \in X$, $k \in \N$, $n \ge m_k$, and $y \in Y$ are such that $\mathsf{d}_n(x,y) \le \delta_{k-1}$, then $\frac{1}{n} \log \| P^{(n)}(x) \| < \lambda_0  + \epsilon_k$.
\end{lemma}

\begin{proof}
We can assume that $k>6$.
Let $\ell \coloneqq m_{k-6}$. 
Perform Euclidean division:
\begin{equation}
n = q \ell + r \, , \quad q>0 \, , \quad 0 \le r < \ell \, .
\end{equation}
For $j \in \ldbrack 0, q \rdbrack$, let $x_j \coloneqq T^{j \ell}(x)$.
Consider the factorization
\begin{equation}\label{e.Euclid_strikes_again}
P^{(n)}(x) = P^{(r)}(x_q) P^{(\ell)}(x_{q-1}) \cdots P^{(\ell)}(x_1) P^{(\ell)}(x_0) \, .
\end{equation}
We will estimate the norm of each factor, starting with 
\begin{equation}\label{e.zzz}
\|P^{(r)}(x_q)\| < e^{Cr} < e^{C\ell} = e^{Cm_{k-6}}  < e^{\epsilon_k m_k} \le e^{\epsilon_k n} \, .
\end{equation}
To bound the other factors, fix $j \in \ldbrack 0, q-1 \rdbrack$.
Note the telescopic sum
\begin{equation}
P^{(\ell)}(x_j) - B^{(\ell)}(x_j) = \sum_{i=0}^{\ell-1} P^{(\ell-i-1)}(T^{i+1} x_j) \left[ P(T^i x_j) - B(T^i x_j) \right] B^{(i)}(x_j) \, .
\end{equation}
Taking norms,
\begin{align}
\left\| P^{(\ell)}(x_j) - B^{(\ell)}(x_j) \right\|  
&\le e^{C(\ell-1)}\sum_{i=0}^{\ell-1} \left\| P(T^i x_j) - B(T^i x_j) \right\| \\ 
&\le e^{C\ell}\sum_{i=0}^{\ell-1} \left\| H(T^i x_j) - \mathrm{Id} \right\| \\ 
&\le e^{C\ell}\sum_{i=0}^{\ell-1} C \theta (T^i x_j) \, . \label{e.ugly}
\end{align}
Since $\mathsf{d}(T^i x_j, Y) \le \mathsf{d}_n(x,y) \le \delta_{k-1}$, using \eqref{e.phi_anchors} we get $\theta (T^i x_j) \le \phi(\delta_{k-1}) = e^{-m_{k-5}}$. Since $\ell = m_{k-6} = \epsilon_k m_{k-5}$, the quantity in \eqref{e.ugly} is less than $1$ for large enough $k$.
In particular,
\begin{equation}
\| P^{(\ell)}(x_j) \| < 1 + \| B^{(\ell)}(x_j) \| \le 2 \| B^{(\ell)}(x_j) \| 
\end{equation}
for all sufficiently large $k$
Let $y_j \coloneqq T^{j \ell}(y)$.
By \eqref{e.twice_norm}, we have 
$\| B^{(\ell)}(x_j) \| < 2 \| B^{(\ell)}(y_j) \|$.
On the other hand, limit~\eqref{e.SUSAET} gives $\| B^{(\ell)}(y_j) \| < e^{(\lambda_0 + \epsilon_k)\ell}$.
It follows that $\| P^{(\ell)}(x_j) \| < e^{(\lambda_0 + \epsilon_k)\ell}$.
Using this bound together with \eqref{e.zzz} in the factorization \eqref{e.Euclid_strikes_again}, the outcome is $\|P^{(n)}(x) \| < e^{(\lambda_0 + \epsilon_k)n}$.
This concludes the proof of \cref{l.trouxa}, and therefore \cref{l.key_estimate} is now unconditionally proved.
\end{proof}

\subsection{Invariant bundle}

The goal of this subsection is showing that the restriction of the cocycle $(T,P)$ to the noncompact set $X \setminus Y$ admits an invariant continuous field of directions (\cref{l.bundle}). The proof is independent of \cref{l.key_estimate}, and its key ingredients are positivity of $P$ and the growth bound \eqref{e.phi_explosion}. 

We will make use of Hilbert's projective metric, which we now recall. Proof of the facts listed below can be found in \cite[{\S}2]{Birkhoff_filho}.
Consider the open positive cone $\R_\smallplus^2 = \R_\smallplus \times \R_\smallplus$, i.e., the interior of the first quadrant in the Cartesian plane. 
For each $v = (v_1,v_2) \in \R_\smallplus^2$, the ray $\R_\smallplus v$ is denoted by $\vec{v}$ or $[v_1 : v_2]$.
Let $\mathbb{P}^1_\smallplus$ denote the set of such rays. 
We metrize this set with the \emph{Hilbert metric}
\begin{equation}
\mathsf{d_H} \big([v_1 : v_2],[w_1:w_2]\big)\coloneqq 
\left| \log \frac{v_1 w_2}{v_2 w_1}\right| \, ,
\end{equation}
or equivalently,
\begin{equation}
\mathsf{d_H} \big( [e^t:1], [e^s:1] \big) = |t-s| \, .
\end{equation}

Let $M$ be a nonnegative $2 \times 2$ matrix such that $\mathrm{Ker}(M) \cap \R_\smallplus^2 = \emptyset$. The \emph{projectivization} of $M$ is the map 
$\vec{M} \colon \mathbb{P}^1_\smallplus \to \mathbb{P}^1_\smallplus$ defined by $\vec{M}(\vec{v}) \coloneqq \overrightarrow{Mv}$.
This map is Lipschitz with respect to the Hilbert metric; 
in fact, for all $\vec{v},\vec{w} \in \mathbb{P}^1_\smallplus$ we have
\begin{equation}\label{e.Garrett}
\mathsf{d_H} \big( \vec{M}(\vec{v}), \vec{M}(\vec{w}) \big) \le 
\left(\tanh \frac{\mathfrak{D}(M)}{4} \right) \, \mathsf{d_H} (\vec{v},\vec{w}) \, ,
\end{equation}
where $\mathfrak{D}(M) \in [0,\infty]$ 
denotes the diameter of the image $\vec{M}(\mathbb{P}^1_\smallplus)$ with respect to the Hilbert metric. 
The Lipschitz constant $\tanh \frac{\mathfrak{D}(M)}{4}$ is at most $1$, and is strictly less than $1$ if the matrix has strictly positive entries. Actually,
\begin{equation}\label{e.Diam_formula}
M = \begin{pmatrix} a & b \\ c & d\end{pmatrix} \quad \Rightarrow \quad
\mathfrak{D}(M) = \left| \log \frac{ad}{bc}\right| \, ,
\end{equation}
We also note that, if $M_1$ and $M_2$ are nonnegative matrices, then, as a direct consequence of bound~\eqref{e.Garrett},
\begin{equation}\label{e.funny}
\mathfrak{D}(M_2 M_1) \le \left(\tanh \frac{\mathfrak{D}(M_2)}{4} \right) \mathfrak{D}(M_1) \, .
\end{equation}

Next, we establish the following ``shrinking property'' of the cocycle $(T,P)$:

\begin{lemma}\label{l.diam_to_zero}
For every $x \in X \setminus Y$,
\begin{equation}\label{e.diam_to_zero}
\lim_{n \to +\infty} \mathfrak{D}(P^{(n)}(T^{-n} x)) = 0 \, .
\end{equation}
\end{lemma}

\begin{proof} 
Let $x \in X \setminus Y$. 
Since $P(x)= B(x) H(x)$, 
we have $\mathfrak{D}(P(x)) \le \mathfrak{D}(H(x))$. 
On the other hand, by formula~\eqref{e.Diam_formula} and definition~\eqref{e.H},
$\mathfrak{D}(H(x)) = -2 \log \tanh \theta(x)$. 
Since the function $\theta$ is bounded above, 
we can write $\tanh \theta(x) \ge C^{-1} \theta(x)$  and therefore
\begin{equation}\label{e.1st_bound}
\mathfrak{D}(P(x)) \le C - 2 \log \theta(x) \, .
\end{equation}
The following elementary bounds are valid for all $s > 0$:
\begin{equation}
\log \tanh s = \log \frac{1-e^{-2s}}{1+e^{-2s}} < \log (1-e^{-2s}) < -e^{-2s} \, .
\end{equation}
Note the following consequence:
\begin{equation}\label{e.2nd_bound}
\log \tanh \frac{\mathfrak{D}(P(x))}{4} \le - C \theta(x) \, .
\end{equation}

Since $T$ is biLipschitz, we have, for all $j \ge 1$,
\begin{align}
\mathsf{d}(T^{-j} x , Y) 
&\ge e^{-Cj} \mathsf{d}(x,Y)   \\  
&\ge e^{- \ell_x - Cj}   \, ,
\end{align}
where $\ell_x \coloneqq \max\{0, -\log \mathsf{d}(x,Y)\}$.
Using \eqref{e.theta} and \eqref{e.phi_explosion},
\begin{align}
\theta(T^{-j} x) 
= \phi(\mathsf{d}(T^{-j} x , Y))  
&> \frac{1}{- \log \mathsf{d}(T^{-j} x , Y)} \\ 
&> \frac{1}{\ell_x + Cj} \quad \text{for all $j \ge 1$.} \label{e.harmonic}
\end{align}

Consider the product $P^{(n)}(T^{-n}x) = P(T^{-1} x) \cdots P(T^{-(n-1)} x) P(T^{-n} x)$, where $n>0$.
Applying \eqref{e.funny} recursively, 
\begin{equation}
\mathfrak{D}(P^{(n)}(T^{-n} x)) \le \mathfrak{D}(P(T^{-n} x))  \prod_{j=1}^{n-1} \tanh \frac{\mathfrak{D}(P(T^{-j} x))}{4} 
\end{equation}
Taking log's and using \eqref{e.1st_bound} and \eqref{e.2nd_bound}, we obtain:
\begin{equation}
\log \mathfrak{D}(P^{(n)}(T^{-n} x)) \le \log\big(C - 2 \log \theta(T^{-n} x)\big) -  C \sum_{j=1}^{n-1} \theta(T^{-j} x) \, .
\end{equation}
Now, using \eqref{e.harmonic},
\begin{equation}\label{e.sofrido}
\log \mathfrak{D}(P^{(n)}(T^{-n} x)) \le \log\big(C + 2 \log (\ell_x + Cn) \big) -  C \sum_{j=1}^{n-1} \frac{1}{\ell_x + Cj}  \, .
\end{equation}
We bound the last sum a la Maclaurin--Cauchy: we have
\begin{equation}
\sum_{j=1}^{n-1} \frac{1}{\ell_x + Cj} > \int_1^n \frac{ds}{\ell_x + Cs} > \frac{1}{C} \log(\ell_x + Cn) \, , 
\end{equation}
which diverges to $+\infty$ as $n \to \infty$, and furthermore dominates the other term in the right-hand side of \eqref{e.sofrido}.
This proves that $\mathfrak{D}(P^{(n)}(T^{-n} x))$ converges to $0$.
\end{proof}

\begin{lemma}\label{l.bundle}
There exist continuous functions $u \colon X \setminus Y \to \R^2_\smallplus$ 
and $f \colon X \setminus Y \to \R$ such that 
for all  $x \in X \setminus Y$, we have $\|u(x)\|=1$ and 
\begin{equation}\label{e.u_and_f}
P(x)u(x) = e^{\lambda_0 + f(x)} u(Tx) \, .
\end{equation}
\end{lemma}

\begin{proof}
For each $x \in X \setminus Y$ and $n>0$, let $\mathcal{K}^n_x$ be the closure of 
$\vec{P}^{(n)}(T^{-n} x) (\mathbb{P}^1_\smallplus)$. 
Note that these sets are nested, that is, $\mathcal{K}^{n+1}_x \subseteq \mathcal{K}^n_x$. 
By \cref{l.diam_to_zero}, their Hilbert diameters converge to zero. It follows that $\bigcap_{n>0} \mathcal{K}^n_x$ contains  a single point $\vec{u}(x)$ in $\mathbb{P}^1_\smallplus$, which can be expressed uniquely as $\vec{u}(x) = \overrightarrow{u(x)}$ where $u(x)$ is a positive unit vector. 
Since $\vec{P}(x) \mathcal{K}^n_x = \mathcal{K}^{n+1}_{Tx}$, it follows that $\vec{P}(x) \vec{u}(x) = \vec{u}(Tx)$. 
In particular, equation \eqref{e.u_and_f} uniquely determines the function~$f$.

We are left to check continuity.
Given $x \in X \setminus Y$ and $\xi>0$,
fix $n$ such that the diameter of $\mathcal{K}^n_x$ is less than $\xi$.
Let $V$ denote the $\xi$-neighborhood of $\mathcal{K}^n_x$ with respect to the Hilbert metric. 
Since the matrix  $P^{(n)}(T^{-n} x)$ has strictly positive entries that depend continuously on $x$, if $\tilde{x}$ is sufficiently close to $x$, then $\mathcal{K}^n_{\tilde x} \subseteq V$. 
Therefore $\vec{u}(\tilde x) \in V$, and in particular, $\mathsf{d_H}(\vec{u}(\tilde x), \vec{u}(x)) < 2\xi$. This proves that the map $x \mapsto \vec{u}(x) \in \mathbb{P}^1_\smallplus$ is continuous with respect to the Hilbert metric. Equivalently, $x \mapsto u(x) \in \R_\smallplus^2$ is continuous with respect to the Euclidean metric. Lastly, $f(x) = -\lambda_0 + \log \|P(x)u(x)\|$ is also a continuous function of $x$.
\end{proof}

\subsection{The cohomological equation}

The goal of this subsection is \cref{l.approx_cob}, which will show that the function $f$ from \cref{l.bundle} is an approximate coboundary near $Y$.

\begin{lemma}\label{l.tau}
There exists a continuous function $\tau \colon X \setminus Y \to \R_\smallplus$ such that 
\begin{gather}
\label{e.tau_anchors}
\mathsf{d}(x,Y) \le \delta_k \quad \Rightarrow \quad \tau(x) \ge m_k \, , \\ 
\label{e.tau_is_slow}
\lim_{x \to Y} |\tau(Tx) - \tau(x)| = 0  \, .
\end{gather}
\end{lemma}

Observe that $\tau(x)$ tends to infinity as $x$ approaches $Y$, but this convergence is slow, because $(\delta_k)$ tends to zero much faster than $(m_k)$ grows.

\begin{proof}
Let $\psi \colon [\log 10, \infty) \to [1,\infty)$ be the piecewise affine homeomorphism with anchor points $\psi(\log 10) = 1$ and $\psi(-\log \delta_k) = m_k$. 
Define a continuous positive function on $X\setminus Y$ by
$\tau(x) \coloneqq \psi(-\log\mathsf{d}(x,Y))$.
Property \eqref{e.tau_anchors} is clearly satisfied. 

Since our dynamics $T$ is a biLipschitz homeomorphism, 
\begin{equation}
	\left| - \log \mathsf{d}(Tx,Y) + \log \mathsf{d}(x,Y) \right| \le C \, .
\end{equation}
So property \eqref{e.tau_is_slow} will follow once we check that the Lipschitz constant of the function $\psi$ restricted to the interval $[n,\infty)$ tends to zero as $n \to \infty$.
The derivative of the piecewise linear function $\psi$ between two consecutive anchor points is:
\begin{align}
\frac{m_{k+1}-m_k}{-\log\delta_{k+1} + \log\delta_k} 
&< \frac{m_{k+1}-m_k}{km_{k+1}} \quad \text{(by \eqref{e.page26})} \\
&< \frac{1}{k} \, ,
\end{align}
so the function $\psi$ behaves appropriately.
\end{proof}

Recall that the function $f$ was introduced in \cref{l.bundle}. 
Let $\|f\|_\infty \coloneqq \sup_{X\setminus Y} |f|$; this quantity is finite (despite the domain of $f$ being noncompact), since it is bounded by $-\lambda_0 + \log \|P\|_\infty$.

For each $t \ge 0$ and $x \in X\setminus Y$, we define the \emph{interpolated Birkhoff sum} $S_t(x)$ as follows: $S_0(x) \coloneqq 0$ and, if $t>0$,
\begin{equation}\label{e.interpolated}
S_t(x) \coloneqq f (x) + f(Tx) + \cdots + f(T^{\lfloor t \rfloor-1} x) + (t-\lfloor t \rfloor) f( T^{\lfloor t \rfloor} x) \, .
\end{equation}
Also, define the auxiliary function 
\begin{equation}\label{e.def_Z}
Z_t(x) \coloneqq \sum_{i=0}^\infty \left(1-\frac{i+1}{t}\right)^+ f(T^i x) \, .
\end{equation}

\begin{lemma}\label{l.Z_properties}
The function $(t,x) \mapsto Z_t(x)$ is continuous on $\R_\smallplus \times (X \setminus Y)$.
It solves the following cohomological equation:
\begin{equation}\label{e.trick}
\frac{S_t(x)}{t} = f(x) + Z_t(Tx) - Z_t(x) \, ,
\end{equation}
where $t>0$.
Furthermore, 
for all $t,s>0$, 
\begin{equation}\label{e.Z_Lip}
|Z_s(x) - Z_t(x)| \le \|f\|_\infty \, |s-t| \, .
\end{equation}
\end{lemma}

We remark that \eqref{e.trick} is basically known, at least for integer $t$ (see \cite{MOP}), while property \eqref{e.Z_Lip} seems to be new.

\begin{proof}
To check continuity, note the terms in the sum \eqref{e.def_Z} vanish for $i \ge t-1$, and are continuous with respect to $x$ and $t$. 
	
Note that \eqref{e.interpolated} can be rewritten as:
\begin{equation}
S_t(x) = \sum_{i=0}^\infty  \left[ (t-i)^+ - (t-i-1)^+ \right] f(T^i x) \, .
\end{equation}
So:
\begin{equation}
\frac{S_t(x)}{t} 
= \underbrace{\sum_{i=0}^\infty  \left(1-\frac{i}{t}\right)^+ f(T^i x)}_{f(x) - Z_t(Tx)}
  -  \underbrace{\sum_{i=0}^\infty  \left(1-\frac{i+1}{t}\right)^+  f(T^i x)}_{Z_t(x)}  \, ,
\end{equation}
which is equation \eqref{e.trick}.

Since the coefficient of $f(T^i x)$ in the series \eqref{e.interpolated}
is increasing with respect to $t$, we have $|Z_s(x) - Z_t(x)| \le \|f\|_\infty |h(s)-h(t)|$, where 
\begin{equation}
h(t) \coloneqq \sum_{i=0}^\infty  \left(1-\frac{i+1}{t}\right)^+ \, .
\end{equation}
This is a continuous function. If $t>0$ is not an integer, then $h$ is differentiable at $t$ with
\begin{equation}
h'(t) = \sum_{i=0}^{\lfloor t \rfloor-1}\frac{i+1}{t^2} = \frac{\lfloor t \rfloor \, (\lfloor t \rfloor+1)}{2t^2} \in [0,1] \, .
\end{equation}
So $h$ has a Lipschitz constant $1$, and inequality \eqref{e.Z_Lip} follows.
\end{proof}

\begin{lemma}
We have
\begin{equation}\label{e.repackaged_key}
\lim_{x \to Y} \frac{S_{\tau(x)}(x)}{\tau(x)} = 0 \, .
\end{equation}
\end{lemma}

\begin{proof}
Given $x \in X \setminus Y$, let $k$ be such that $\delta_{k+1} \le d(x,Y) \le \delta_k$. 
Then \eqref{e.tau_anchors} implies that $m_k \le \tau(x) \le m_{k+1}$.
Let $n \coloneqq \lfloor \tau(x) \rfloor$, so $m_k \le n \le m_{k+1}$, that is, conditions \eqref{e.two_ranges} are met.
Now write $\tau(x) = n + r$, where $0 \le r < 1$.
By relations \eqref{e.interpolated} and \eqref{e.u_and_f},
\begin{align}
\frac{S_{\tau(x)}(x)}{\tau(x)} 
&= -\lambda_0 + \frac{1}{n+r} \log \|P^{(n)}(x)u(x)\| +  \frac{r}{n+r} \log \|P(T^n x)u(T^n x)\| \\ 
&= -\lambda_0 + \frac{1}{n} \log \|P^{(n)}(x)u(x)\| + O(\tfrac{1}{n}) \, .
\end{align}
If $x$ is close to $Y$, then $k$ and $n$ are large, and \cref{l.key_estimate} ensures that this quantity is small.
\end{proof}

\begin{lemma}[Approximate coboundary] \label{l.approx_cob}
There exist two continuous functions 
{$g \colon X \setminus Y \to \R $} and {$r \colon X \to \R$}
such that
\begin{equation}\label{e.approx_coboundary}
	f(x) = g(Tx) - g(x) + r(x) \quad \text{for all } x \in X \setminus Y
\end{equation}
and $r$ vanishes on $Y$. 
\end{lemma}

\begin{proof}
We let $g(x) \coloneqq -Z_{\tau(x)} (x)$.
This is a continuous function on $X \setminus Y$ (see \cref{l.Z_properties}).
Now define $r$ on $X \setminus Y$ by \eqref{e.approx_coboundary}.
Then we  have
\begin{align}
|r(x)| 
&= \left| f(x) + g(x) - g(Tx)\right|  \\
&= \left| f(x) - Z_{\tau(x)} (x) + Z_{\tau(Tx)} (Tx) \right| \\
&\le \left| f(x) - Z_{\tau(x)} (x) + Z_{\tau(x)} (Tx) \right|
+  \big|Z_{\tau(Tx)} (Tx) - Z_{\tau(x)} (Tx) \big| 
\end{align}
Therefore
\begin{alignat}{2}
|r(x)|
&\le \left| \frac{S_{\tau(x)}(x)}{\tau(x)} \right| + \|f\|_\infty \left|\tau(Tx) - \tau(x)\right| &\quad &\text{(by \eqref{e.trick} and \eqref{e.Z_Lip})} \\
&\to 0 \text{ as $x \to Y$} &\quad&\text{(by \eqref{e.repackaged_key} and \eqref{e.tau_is_slow}).}
\end{alignat}
This allows us to define $r$ as $0$ on $Y$ and obtain a continuous function on $X$.
\end{proof}

\subsection{The final perturbation}

For each $x \in X \setminus Y$, let $S(x)$
be the matrix uniquely defined by the following conditions:
\begin{equation}
S(x) u(x) =  e^{r(x)} u(x)  \quad \text{and} \quad
S(x) v =  e^{-r(x)} v  \text{ if }  v \perp u(x) \, .
\end{equation}
Since $u(x)$ and $r(x)$ depend continuously on $x$, so does $S(x)$.
Furthermore, $S(x)$ is a symmetric matrix with eigenvalues $e^{\pm r(x)}$.
It follows that 
\begin{equation}
	\| S(x) - \mathrm{Id} \| = e^{|r(x)|} - 1 \, .
\end{equation}
Since the continuous function $r$ vanishes on $Y$, defining $S(x)$ as the identity matrix for all $x \in Y$, we obtain a continuous map $S \colon X \to \mathrm{SL}(2,\R)$.
Finally, define $A \colon X \to \mathrm{SL}^{\smallpm}(2,\R)$ by
\begin{equation}\label{e.final_perturb}
A(x) = P(x)S(x) \, .
\end{equation}
Note that
\begin{equation}\label{e.u_and_g}
A(x)u(x) = e^{\lambda_0 + g(Tx) - g(x)} u(Tx) \, , \quad \text{for all  $x \in X \setminus Y$.}
\end{equation}
The cocycle $(T,A)$ is not uniformly hyperbolic, since its restriction to the invariant set $Y$ is still the Walters cocycle. 
Let us show that the cocycle is monochromatic:

\begin{lemma}
$\lambda(T,A,\mu) = \lambda_0$ for every ergodic measure $\mu$.
\end{lemma}

\begin{proof}
Let $\mu$ be a $T$-invariant ergodic Borel probability measure on $X$.
If $\mu(Y) = 1$, then $\mu$ equals $\nu$, the unique invariant measure for the Veech-like map, so the Lyapunov exponent is $\lambda_0$ by definition. 
So we can assume that $\mu(Y) = 0$.

For every $x \in X \setminus Y$ and $n \ge 0$, iteration of \eqref{e.u_and_g} gives
\begin{equation}
A^{(n)}(x) u(x) = e^{n\lambda_0 + g(T^n x) - g(x)} u(T^n(x)) \, .
\end{equation}
By the Poincar\'e recurrence theorem, for $\mu$-almost every $x$ there exists a sequence $n_i \to +\infty$ such that $T^{n_i}x \to x$.
We can assume that $x \not\in Y$, so
\begin{equation}
\frac{1}{n_i} \log \| A^{(n_i)}(x) u(x) \| = \lambda_0 + \frac{g(T^n x) - g(x)}{n_i} \to \lambda_0 \quad \text{as } i \to \infty \, ,
\end{equation}
since $g$ is continuous on the open set $X \setminus Y$.
On the other hand, by Oseledets theorem,  for $\mu$-almost every $x$ and every non-zero vector $v$ we have 
\begin{equation}
\frac{1}{n} \log \| A^{(n_i)}(x) v \| \to \pm \lambda(T,A,\mu) \quad \text{as } n \to \infty \, .
\end{equation}
It follows that $\lambda(T,A,\mu) = \lambda_0$. 
\end{proof}

This concludes the proof of \cref{t.main_positive}.
Observe that it is not necessarily the case that matrices $A(x)$ are positive (or even nonnegative) for $x \in X \setminus Y$. Nevertheless, the cocycle restricted to $X \setminus Y$ admits a continuous field of lines such along which the log of the rate of expansion is cohomologous to a constant. In particular, the restriction of the cocycle to any invariant compact set $Z \subset X \setminus Y$ is uniformly hyperbolic. 

\section{Construction of a symbolic Veech-like map}\label{s.symbolic_Veech}

In this section we provide a proof of \cref{t.existence_Veech} by showing an explicit Veech-like shift map $T$. We use a variation of a construction explained to us by Scott Schmieding. We also provide an explicit Walters cocycle with base $T$.

\subsection{Definition of the shift space}

Consider the following \emph{alphabet} 
\begin{equation}
\mathsf{A} \coloneqq  \{ \mathord{\uparrow}, \, \mathord{\downarrow} , \,  0 \} \, ,
\end{equation}
whose elements are called \emph{letters}.
Finite (possibly empty) strings of letters are called \emph{words}. 
The set $\mathsf{A}^*$ of all words is a semigroup under the concatenation operation. 
We say that a word $w'$ is a \emph{subword} of a word $w$
if $w = u w' v$ for some (possibly empty) words $u$ and $v$. 
If $w = uv$, then $u$ is called a \emph{prefix} of $w$ and $v$ is called a \emph{suffix} of $w$.
The number of letters of a word $w$ is called its \emph{length}, and is denoted $|w|$.
If $w$ is a word, let $\overline{w}$ denote the \emph{conjugate} word obtained by interchanging letters $\mathord{\uparrow}$ and $\mathord{\downarrow}$.

We fix an integer sequence $(m_k)$ such that 
\begin{equation}\label{e.mk}
m_k \ge 2 \quad \text{and} \quad \sum_{k=1}^\infty \frac{1}{m_k} < \infty \, .
\end{equation}
Then we define recursively a sequence of words $e_1, e_2, \dots$ as follows:
\begin{align}
\label{e.first_word}
e_{1}	&\coloneqq \mathord{\uparrow} \mathord{\downarrow}  \, ,  
\\ 
\label{e.next_word}
e_{k+1} &\coloneqq 
e_k^{m_k} \, 0 \, e_k \, 
\overline{e_k}^{m_k} \,  0 \, \overline{e_k} \, .
\end{align}

Let $Y$ be the set of doubly-infinite sequences $\omega \in \mathsf{A}^\Z$ such that for every $n > 0$, the word $\omega_{\ldbrack -n,n\rdbrack}$ is a subword of some $e_k$.
Equipping $\mathsf{A}^\Z$ with the product topology, the subset $Y$ is closed and shift-invariant. 
The fact that $Y$ is nonempty follows from the observation that for every $k$, the word $e_k$ can be extended both to the left and to the right to form the word $e_{k+1}$.
Let $T \colon Y \to Y$ be the restriction of the shift transformation. 
We will show: 

\begin{theorem}\label{t.our_Veech}
The map $T$ defined above is Veech-like, and has zero topological entropy.
\end{theorem}

The theorem above obviously implies \cref{t.existence_Veech}. 
We proceed with the proof, which is divided into several steps.

\subsection{Entropy and minimality}

Let $\mathcal{L} = \mathcal{L}(Y)$ be the \emph{language} of $Y$, that is, the set of all words that occur in elements of $Y$. Equivalently, $w \in \mathcal{L}$ if and only if $w$ is a subword of some $e_k$.

For each $k>0$, the four words
\begin{equation}\label{e.k-elementary}
e_k \, , \quad  \overline{e_k} \, , \quad  e_k \, 0 \, , \quad \text{and} \quad \overline{e_k} \, 0  
\end{equation}
will be called \emph{$k$-elementary words}.

\begin{lemma}\label{l.sofic}
For each $k>0$, let $\Sigma_k \subseteq \mathsf{A}^\Z$ be the shift space formed by all bi-infinite words obtained by concatenation of $k$-elementary words. 
Then $\Sigma_k \supseteq Y$.
\end{lemma}

Before proving the \lcnamecref{l.sofic}, it is convenient to have an alternative description of the shift $\Sigma_k$.
We use the terminology from \cite[{\S}3.1]{LindMarcus}.
Let $\mathcal{G}_k$ be a directed graph consisting of four loops based on a common vertex~$\bullet$ (that is, a ``four-leaf clover''), two of which with length $|e_k|$ and the other two with length $|e_k|+1$. Label the edges in such a way that the label sequences corresponding to the loops are the four $k$-elementary words \eqref{e.k-elementary}. For example, the graph $\mathcal{G}_1$ is depicted in \cref{fig.graph}. Then $\Sigma_k$ is the set of label sequences of infinite walks on $\mathcal{G}_k$. That is, $\Sigma_k$ is the sofic shift presented by $\mathcal{G}_k$. The language $\mathcal{L}(\Sigma_k)$ of the shift $\Sigma_k$ is formed by all label sequences of finite walks in the graph. 

\begin{figure}
	\begin{tikzpicture}[>=latex, font=\footnotesize, scale=1.25]
		\node (zero)	at ( 0, 0) {$\bullet$};
		\node (um) 		at ( 0, 2) {$\circ$};
		\node (dois)	at ( 1, 1) {$\circ$};
		\node (tres)	at ( 2, 0) {$\circ$};
		\node (quatro)	at ( 0,-2) {$\circ$};
		\node (cinco)	at (-1,-1) {$\circ$};
		\node (seis)	at (-2, 0) {$\circ$};
		\draw ( 0,   1) node {$e_1$};
		\draw ( 0,  -1) node {$\overline{e_1}$};
		\draw ( 1, .35) node {$e_1 0$};
		\draw (-1,-.35) node {$\overline{e_1} 0$};
		\draw[->,bend right]	(zero)	to	(um);
		\draw[->] 				(zero)	to	(dois);
		\draw[->,bend left]		(zero)	to	(quatro);
		\draw[->] 				(zero)	to	(cinco);
		\draw[->,bend right]	(um)	to	(zero);
		\draw[->]				(dois)	to	(tres);
		\draw[->]				(tres)	to	(zero);
		\draw[->,bend left]		(quatro)to	(zero);
		\draw[->] 				(cinco) to	(seis);
		\draw[->] 				(seis)	to	(zero);
	\end{tikzpicture}
	\caption{A graph $\mathcal{G}_1$ presenting the shift $\Sigma_1$. Edges going up have label~$\mathord{\uparrow}$, edges going down have label~$\mathord{\downarrow}$, and horizontal edges have label~$0$. The label sequences of the four loops are indicated.}\label{fig.graph}
\end{figure}

\begin{proof}[Proof of \cref{l.sofic}]
We claim that every $\ell$-elementary word with $\ell \ge k$ is the label sequence of a closed walk in the graph $\mathcal{G}_\ell$ starting and ending at the center vertex~$\bullet$. The claim is proved by induction on $\ell$, using the recursive relation \eqref{e.next_word}.
As a consequence of the claim, $e_\ell \in \mathcal{L}(\Sigma_k)$ for every $\ell$. 

Now if $w$ is any word in $\mathcal{L}(Y)$, by definition $w$ is a subword of some $e_\ell$. In particular, $w \in \mathcal{L}(\Sigma_k)$. We have proved that $\mathcal{L}(Y) \subseteq \mathcal{L}(\Sigma_k)$. It follows (see \cite[p.~10]{LindMarcus}) that $Y \subseteq \Sigma_k$.
\end{proof}

\begin{lemma}\label{l.entropy}
The map $T$ has zero topological entropy.
\end{lemma}

\begin{proof}
It is clear that the shift $\Sigma_k$ in \cref{l.sofic} has entropy $h(\Sigma_k) \le \frac{\log 4}{|e_k|}$. Therefore, $h(Y) = 0$.
\end{proof}

\begin{lemma}\label{l.minimalities}
$T$ and $T^2$ are minimal. 
\end{lemma}

Obviously, minimality of $T^2$ implies minimality of $T$, but for the sake of clarity we consider $T$ first. 

\begin{proof}
Recall that a homeomorphism of a compact metric space is minimal if and only if for every nonempty open subset $U$, there exists $n_0>0$ such that every segment of orbit of length $n_0$ intersects $U$ (see \cite[Prop.~1.6.25]{FisherH}).  
Therefore, in the case under consideration, proving that $T$ is minimal is tantamount to showing that given any word $v \in \mathcal{L}$, every sufficiently long word $w \in \mathcal{L}$ contains $v$ as a subword. 
Let us check this property. 
Given $v$, there exists $k>0$ such that $v$ is a subword of either $e_k$ or $\overline{e_k}$. Now suppose that $w \in \mathcal{L}$ has length $|w|>2|e_k|$. 
Then $w$ corresponds to a walk 
of length $>2|e_k|$ on the graph $\mathcal{G}_k$. This walk must traverse one of the four loops. That is, $w$ contains a subword which is $k$-elementary. In particular, $w$ has $v$ as a subword. 
We conclude that $T$ is minimal.

In order to prove that $T^2$ is minimal, it is sufficient to check that given any word $v \in \mathcal{L}$, every sufficiently long word $w \in \mathcal{L}$ contains $v$ as a subword in both odd and even positions. The argument is similar: Given $v \in \mathcal{L}$, take $k$ such that $v$ is a subword of either $e_k$ or $\overline{e_k}$. Now let $w \in \mathcal{L}$ have  length $|w|>2|e_{k+1}|$. Then $w$ contains a subword $f$ which is one of the four $k+1$-elementary words. In any case, $e_k$ appears in $f$ in both odd and even positions. In particular, $v$ occurs as a subword of $w$ in both odd and even positions. 
This shows that $T^2$ is minimal.
\end{proof}

\subsection{Unique ergodicity of $T$}

Given two words $v$ and $w$ in $\mathsf{A}^*$, where $w$ is nonempty, let $|w|_v$ denote the number of occurrences of $v$ as a subword of $w$, that is, the number of possible ways of expressing $w$ as a concatenation $u v u'$. Note that 
\begin{equation}\label{e.silly}
|w|_v \le \max \left( 0, |w| - |v| + 1\right) \le |w| \, .
\end{equation}

\begin{lemma}\label{l.N_bounds}
For all words $v, w_1, \dots, w_p$ in $\mathsf{A}^*$, with $v \neq \emptyset$, we have
\begin{equation}\label{e.N_bounds}
\sum_{i=1}^p |w_i|_v \le |w_1\cdots w_p|_v \le \sum_{i=1}^p |w_i|_v  \ + \ (p-1)(|v|-1) \, .
\end{equation}
\end{lemma}

\begin{proof}
We claim that $|w_1 w_2|_v$ is at least $|w_1|_v + |w_2|_v$ and at most $|w_1|_v + |w_2|_v + |v| - 1$. Indeed, the lower bound corresponds to ``internal'' occurrences of $v$ as a subword of either $w_1$ or $w_2$, while the upper bound takes into account possible ``concatenation'' occurrences. So the \lcnamecref{l.N_bounds} is true in the case $p=2$ (and tautological in the case $p=1$). An induction on $p$ concludes the proof. 
\end{proof}

For every nonempty word $v$ in the language $\mathcal{L}$, we will associated a ``limit frequency'' $\Phi(v)$, as follows:

\begin{lemma}\label{l.freq_e}
For any nonempty $v \in \mathcal{L}$,
the sequences $\frac{|e_k|_v}{|e_k|}$ and $\frac{|\overline{e_k}|_v}{|e_k|}$
converge to common limit $\Phi(v)$, which is strictly positive.
\end{lemma}

\begin{proof}
Fix a nonempty word $v \in \mathcal{L}$. 
Applying \cref{l.N_bounds} to the recursive formula \eqref{e.next_word}, we obtain
\begin{equation}
(m_k + 1) (|e_k|_v + |\overline{e_k}|_v) \le |e_{k+1}|_v \le (m_k + 1) (|e_k|_v + |\overline{e_k}|_v) + (2m_k + 3) (|v|-1) \, .
\end{equation}
In particular, 
\begin{equation}
|e_{k+1}|_v = (m_k+1)(|e_k|_v + |\overline{e_k}|_v) + O(m_k) \, .
\end{equation}
Similarly, 
\begin{equation}
|\overline{e_{k+1}}|_v = (m_k+1)(|e_k|_v + |\overline{e_k}|_v) + O(m_k) \, .
\end{equation}
Let $S_k \coloneqq |e_k|_v + |\overline{e_k}|_v$ and $D_k \coloneqq |e_k|_v - |\overline{e_k}|_v$. 
Then 
\begin{align}
S_{k+1} &= 2(m_k+1)S_k + O(m_k) \quad \text{and} \label{e.next_sum}
\\
D_{k+1} &= O(m_k) \, . \label{e.spidey}
\end{align}
There exists $k_0$ such that $S_{k_0}>0$. 
If $k \ge k_0$, then $S_k > 2^{k-k_0} S_{k_0}$ and 
\begin{align}
\frac{S_{k+1}}{S_k} 
= 2(m_k+1) + O \left(\frac{m_k} {S_k} \right)
&= 2(m_k+1) \left(1 + O \left(\frac{1}{S_k} \right)\right) \\ 
&= 2(m_k+1) \left(1 + O(2^{-k})\right) \, . \label{e.Saturday}
\end{align} 
On the other hand, by the recursion \eqref{e.next_word}, 
\begin{equation}\label{e.len_rec}
|e_{k+1}| = 2(m_k+1)|e_k|+2 \, .
\end{equation}
So $|e_k| \ge 2^k$ and 
\begin{equation}
\frac{|e_{k+1}|}{|e_k|} = 2(m_k+1)\left(1 + O(2^{-k})\right) \, .
\end{equation}
Using \eqref{e.Saturday} we obtain that for all $k \ge k_0$,
\begin{equation}
\frac{S_{k+1}}{|e_{k+1}|} = \frac{S_k}{|e_k|} \,  r_k \, , \quad 
\text{where } r_k >0  \text{ and } r_k = 1+ O(2^{-k}) \, .
\end{equation}
It follows that $\frac{S_k}{|e_k|} = \frac{S_{k_0}}{|e_{k_0}|} \prod_{j=k_0}^k r_j$  converges as $k \to \infty$ to a finite non-zero limit, which we call $2\Phi(v)$.

Next, by \eqref{e.spidey} and \eqref{e.len_rec}, we have 
$\frac{D_{k+1}}{|e_{k+1}|} = O \left( \frac{1}{|e_k|} \right) = o(1)$.
Therefore the two sequences 
\begin{equation}
\frac{|e_k|_v}{|e_k|}
=\frac{1}{2}\left(\frac{S_k}{|e_k|}+\frac{D_k}{|e_k|}\right) 
\quad \text{and} \quad 
\frac{|\overline{e_k}|_v}{|e_k|}
=\frac{1}{2}\left(\frac{S_k}{|e_k|}-\frac{D_k}{|e_k|}\right)
\end{equation}
converge to the same positive number $\Phi(v)$, as we wanted to show.
\end{proof}

The next lemma builds on \cref{l.freq_e}:

\begin{lemma}\label{l.freq_w}
For every nonempty word $v \in \mathcal{L}$ and every $\varepsilon>0$, there exists $n>0$ such that if $w \in \mathcal{L}$ and $|w|>n$, then $\frac{|w|_v}{|w|} = \Phi(v) + O(\varepsilon)$.
\end{lemma}

\begin{proof}
Fix a nonempty word $s \in \mathcal{L}$ and a positive number $\varepsilon$.
Using \cref{l.freq_e}, we see that if $k$ is large enough, we have
\begin{equation}\label{e.coca_cola}
\left| \frac{|u|_v}{|f|} - \Phi(v) \right| < \varepsilon \quad \text{for each $k$-elementary word $f$.}
\end{equation} 
Fix such a large enough $k$ with $|v|< \varepsilon |e_k|$.
Then choose an integer $n > \varepsilon^{-1} (|e_k|+1)$, and let $w$ be any word in $\mathcal{L}$ with length at least $n$.
As a consequence of \cref{l.sofic}, the word $w$ admits a factorization
\begin{equation}\label{e.factors_w}
	w = f_1 f_2 \cdots f_{p-1} f_p
\end{equation}
where the inner factors $f_2, \dots, f_{p-1}$ are $k$-elementary words,
and the extremal factors $f_1$ and $f_p$ have length at most $|e_k|$.
Note that
$|w| \ge (p-2)|e_k|$.
Applying \cref{l.N_bounds} to the factorization \eqref{e.factors_w}, we obtain 
\begin{equation}
\sum_{i=2}^{p-1} |f_i|_v \le |w|_v 
< \sum_{i=1}^p |f_p|_v +  p|v|  \, .
\end{equation}
The difference between the upper and the lower bound is
\begin{align}
|f_0| + |f_1| + p|v|
< 2|e_k| + p \varepsilon |e_k|  
\le 2|e_k| + \varepsilon (|w| - 2|e_k|)
= O(\varepsilon |w|).
\end{align}
For each $i$ between $2$ and $p-1$, by \eqref{e.coca_cola} we have $|f_i|_v = \Phi(v) |f_i| + O(\varepsilon |f_i|)$.
Also,
\begin{equation}
\sum_{i=2}^{p-1} |f_i| = |w| - |f_0| - |f_1| = |w| + O(\varepsilon |w|) \, .
\end{equation}
It follows that
\begin{equation}
|w|_v 
= \sum_{i=2}^{p-1} |f_i|_v + O(\varepsilon |w|) 
= \Phi(v) |w| + O(\varepsilon |w|) \, ,
\end{equation}
as we wanted to show. 
\end{proof}

\begin{lemma}\label{l.ue}
$T$ is uniquely ergodic.
\end{lemma}

\begin{proof}
Given a nonempty word $v$ in $\mathcal{L}$, consider the cylinder
${}_0 [v]$ consisting of all $\omega \in Y$ such that $\omega_{\ldbrack 0, |v|-1 \rdbrack} = s$.
It follows from \cref{l.freq_w} that the Birkhoff averages of the characteristic function of this cylinder converge uniformly to a constant, namely $\Phi(v)$. 
The same is true for more general cylinders ${}_n [v] \coloneqq T^{-n}({}_0 [v])$.
The set of linear combinations of characteristic functions of cylinders is dense in $C^0(Y,\R)$.
Therefore (see \cite[Proposition~4.1.15]{KatokH}), $T$ is uniquely ergodic. 
\end{proof}

\subsection{Non-unique ergodicity of $T^2$}

Let us define two auxiliary functions $\varphi$ , $\psi$ on $Y$ as follows:
\begin{align}
\label{e.def_phi}
\varphi(\omega) &\coloneqq
\begin{cases}
\phantom{-}1	&\text{if } \omega_0 = \mathord{\uparrow} \, , \\ 
-1 				&\text{if } \omega_0 = \mathord{\downarrow} \, , \\
\phantom{-}0	&\text{if } \omega_0 = 0 \, ,
\end{cases}
\\
\label{e.def_psi}
\psi &\coloneqq \varphi - \varphi \circ T \, .
\end{align}
If we identify the letters $\mathord{\uparrow}$ and $\mathord{\downarrow}$ with the numbers $1$ and $-1$, respectively, then
\begin{equation}
\varphi(\omega) = \omega_0 \quad \text{and} \quad
\psi(\omega) = \omega_0 - \omega_1 \, .
\end{equation}

We define another auxiliary function $\theta \colon \mathcal{L} \to \R$ as follows: $\theta(\emptyset) = 0$ and, if $w \neq \emptyset$,
\begin{equation}
\theta(w) \coloneqq \sum_{i=0}^{|w|-1} (-1)^i  \varphi(T^i(\omega)) \, ,
\end{equation}
where $\omega$ is any element of the cylinder ${}_0 [w]$.
So, in terms of the letter identifications $\mathord{\uparrow} = 1$ and $\mathord{\downarrow} = -1$,
if the word $w$ is spelt $a_0 \cdots a_{n-1}$, then $\theta(w) =  \sum_{i=0}^{n-1} (-1)^i a_i$.
The function $\theta$ has the following properties:
\begin{equation}\label{e.theta_properties}
\theta(\overline{w}) = - \theta(w) \, , \quad
\theta(w v) = \theta(w) + (-1)^{|w|} \theta(v) 
\end{equation}

\begin{lemma}\label{l.c_freq}
The sequence $\frac{\theta(e_k)}{|e_k|}$ converges as $k \to \infty$ to a number $c>0$. 
\end{lemma}

\begin{proof}
Using the definitions \eqref{e.first_word}, \eqref{e.next_word} and the properties \eqref{e.theta_properties}, we obtain:
\begin{equation}
\theta(e_1)=2  \quad \text{and} \quad \theta(e_{k+1}) = 2 (m_k - 1) \theta(e_k) \, .
\end{equation}
Using the recursive relation \eqref{e.len_rec}, we obtain
\begin{equation}
\frac{\theta(e_{k+1})}{|e_{k+1}|} = \frac{\theta(e_k)}{|e_k|} \, r_k \quad \text{where} \quad 
r_k = \frac{m_k-1}{m_k+1+|e_k|^{-1}} \, . 
\end{equation}
Since $m_k \ge 2$ and $|e_k| \ge 2$, we have
\begin{equation}
0 < \frac{m_k - 1}{m_k + 1.5} \le r_k < 1 
\end{equation}
and 
\begin{equation}
\sum_{k=1}^\infty (1-r_k) \le \sum_{k=1}^\infty \frac{2.5}{m_k + 1.5} < \infty \, ,
\end{equation}
by the convergence condition in \eqref{e.mk}.
If follows that $c \coloneqq \lim \frac{\theta(e_k)}{|e_k|}$ exists and is positive. 
\end{proof}

\begin{lemma}\label{l.square_not_ue}
$T^2$ is not uniquely ergodic.
\end{lemma}

\begin{proof}
By \cref{l.ue}, $T$ admits a unique invariant probability measure, say $\nu$.
Consider the function $\psi$ defined in \eqref{e.def_psi}.
Since this is a coboundary with respect to $T$, we have $\int \psi \, d\nu = 0$.
If $T^2$ were uniquely ergodic, then its invariant measure would be $\nu$,
and therefore the Birkhoff averages of $\psi$ with respect to $T^2$ would converge uniformly to $0$.
On the other hand, for each $k>0$ we can find a point $\omega$ in the cylinder ${}_0 [e_k]$. Let $n \coloneqq |e_k|/2$ and consider the Birkhoff average
\begin{equation}
\frac{1}{n} \sum_{i=0}^{n-1} \psi(T^{2i} (\omega)) 
= \frac{2 \theta(\omega_{\ldbrack 0,|e_k|-1 \rdbrack})}{|e_k|} 
= \frac{2 \theta(e_k)}{|e_k|} \, .
\end{equation}
By \cref{l.c_freq}, this quantity is close to $2c \neq 0$ if $k$ is large enough. 
So $T^2$ cannot be uniquely ergodic. 
\end{proof}

The combination of \cref{l.entropy,l.minimalities,l.ue,l.square_not_ue} yields \cref{t.our_Veech}.

\subsection{Appendix: An explicit Walters cocycle}\label{ss.appendix}

We have explicitly constructed a Veech-like map $T$ in \cref{t.our_Veech}.
Then \cref{t.Walters} guarantees the existence of a Walters cocycle with base $T$.
However, the proof of the latter \lcnamecref{t.Walters} is slightly indirect. 
For the sake of completeness, let us provide an explicit example of Walters cocycle with base $T$.

\begin{theorem}\label{t.our_Walters}
Let $T \colon Y \to Y$ be the Veech-like map constructed above. 
Let $B \colon Y \to \mathrm{SL}^{\smallpm}(2,\R)$ be defined by 
\begin{equation}\label{e.our_Walters}
B(\omega) = \begin{pmatrix} 0 & e^{\varphi(\omega)} \\ e^{-\varphi(\omega)} & 0 \end{pmatrix} \, , 
\end{equation}
where $\varphi$ is as in \eqref{e.def_phi}. 
Then $(T,B)$ is a Walters cocycle, and its Lyapunov exponent with respect to the unique invariant measure equals the number $c$ from \cref{l.c_freq}.
\end{theorem}

The proof needs the following observation:

\begin{lemma}\label{l.no_other_limits}
Let $\omega \in Y$ be such that the limits
\begin{equation}\label{e.two_limits}
\lim_{n \to \infty} \frac{1}{n} \left| \sum_{j=0}^{n-1} (-1)^j \varphi(T^j (\omega)) \right| 
\quad \text{and} \quad
\lim_{n \to \infty} \frac{1}{n} \left| \sum_{j=-n}^{-1} (-1)^j \varphi(T^j (\omega)) \right|
\end{equation}
both exist.
Then at least one of them equals $c$.
\end{lemma}

\begin{proof}
Let $\omega \in Y$, $\varepsilon>0$, and $n_0 > 0$ be arbitrary.
We will show that there exists $n > n_0$ such that either $\frac{1}{n} \left| \sum_{j=0}^{n-1} (-1)^j \varphi(T^j (\omega)) \right|$ or $\frac{1}{n}  \left| \sum_{j=-n}^{-1} (-1)^j \varphi(T^j (\omega)) \right|$ is $c + O(\varepsilon)$. Once we do that, the \lcnamecref{l.no_other_limits} will be proved. 

Choose $k>0$ satisfying the following conditions:
\begin{equation}\label{e.k_conditions}
|e_{k+1}| > 2n_0 \, , \quad
m_k > \frac{1}{\varepsilon} \, , \quad \text{and} \quad 
\left| \frac{\theta(e_k)}{|e_k|} - c \right| < \varepsilon \, .
\end{equation}
By \cref{l.sofic}, $\omega \in \Sigma_{k+1}$, that is, the bi-infinite word $\omega$ can be written as a concatenation $\cdots f_{-1} f_0 f_1 \cdots$ of $k+1$-elementary words $f_i$, where the zeroth position is somewhere in $f_0$. Write $f_0 = \omega_{\ldbrack -p,q-1 \rdbrack}$, where $p\ge 0$ and $q \ge 1$. 
Let
\begin{equation}
n \coloneqq \max(p,q)
\quad \text{and} \quad 
w \coloneqq 
\begin{cases}
\omega_{\ldbrack -p,-1 \rdbrack} &\text{if } n=p \, , \\
\omega_{\ldbrack 0,q-1 \rdbrack} &\text{if } n=q \, . 
\end{cases}
\end{equation}
Since $p + q = |f_0|$, 
we have 
\begin{equation}\label{e.at_least_half}
|w| = n \ge \frac{|f_0|}{2} \ge \frac{|e_{k+1}|}{2} > n_0 \, . 
\end{equation}
To complete the proof, we will check that 
$\frac{|\theta(w)|}{|w|} = c + O(\varepsilon)$. 

Consider the case $n=p$.
Then $w$ is a prefix of the $k+1$-elementary word $f_0$.
Since $\theta(\overline{w}) = - \theta(w)$, 
we can assume that $f_0$ is either $e_{k+1}$ or $e_{k+1} 0$.
Since $|w| \ge \frac{1}{2} |f_0|$, it follows from the recursive definition \eqref{e.next_word} that, by either adding or removing at most $|e_k|+2$ letters at the end of the word $w$, we can obtain a word of the form
\begin{equation}
\tilde{w} = e_k^{m_k} \, 0 \, e_k \, \overline{e_k}^r \, , \quad \text{where } r \ge 0 \, .
\end{equation}
Then 
\begin{equation}
|w| = 
(m_k+r)|e_k| + O(|e_k|) \, .
\end{equation}
Also, by properties~\eqref{e.theta_properties} we have 
$\theta(\tilde{w}) = (m_k-1+r)\theta(e_k)$ and 
$\big| \theta(w) - \theta(\tilde{w}) \big| \le |e_k| + 2$, so 
\begin{equation}
\theta(w) = (m_k+r) \theta(e_k) + O(|e_k|) \, .
\end{equation}
Then, using a standard Calculus estimate,
\begin{equation}
\frac{\theta(w)}{|w|} 
= \frac{\theta(e_k)}{|e_k|} + O \left(\frac{1}{m_k+r}\right) 
= c + O(\varepsilon) \, ,
\end{equation}
as we wanted to show. 
This completes the argument in the case $n=p$.
The case $n=q$ is handled in a similar manner. 
\end{proof}

\begin{proof}[Proof of \cref{t.our_Walters}]
Let $\nu_0$ and $\nu_1$ be as in \cref{l.double_ergodicity}.
Applying the ergodic theorem to the map $T^2$, the invariant measure $\nu_0$, and the function  $\psi = \varphi - \varphi \circ T$, we ensure the existence of points $\omega \in Y$ such that both limits \eqref{e.two_limits} exist and are equal to $\frac{1}{2} \int \psi \, d\nu_0$. 
Therefore $\int \psi \, d\nu_0 = \pm 2c$, by \cref{l.no_other_limits}.
Interchanging $\nu_0$ and $\nu_1$ if necessary, we can assume that the sign above is a $+$.
On the other hand, since $T_* \nu_0 = \nu_1$, we have
$\int \varphi \, d\nu_0 - \int \varphi \, d\nu_1 = \int \psi \, d\nu_0 = 2c$.
Then the proof of \cref{t.Walters} (formula \eqref{e.lambda}, actually) shows that \eqref{e.our_Walters} defines a Walters cocycle with Lyapunov exponent $c$. 
\end{proof}

\section{Questions}

In this paper, we showed that the hypothesis of H\"older regularity cannot be dropped from the results of \cite{DWG,Guy}. Nevertheless, we do not know the exact regularity threshold for the validity of such theorems. 

\begin{question}
What is the optimal modulus of regularity for a map $A$ as in \cref{t.main_SFT}?
\end{question}

To estimate the modulus of regularity of the maps $A$ arising from our construction, the first step is to analyze the growth of the time scales $m_k$ introduced in \cref{l.time_scale}. These appear to depend intricately on the choice of the Walters cocycle and especially on the underlying Veech-like map. While the remainder of our construction is largely explicit, carefully keeping track of all moduli of regularity would nonetheless involve considerable effort.

\medskip

A motivation for works such as DeWitt--Gogolev \cite{DWG} comes from the problem of \emph{periodic data rigidity}: to what extent do the cocycle products along periodic orbits determine the cohomology of the cocycle? We note that this general problem is  unsolved  for H\"older cocycles over hyperbolic dynamics: see \cite[{\S}13.4.1]{Sad24} for discussion. Under this perspective, the following question, suggested by one of the referees, is natural. Following \cite[Def.~2.3]{DWG}, we say that a linear cocycle has \emph{constant periodic data} if all probability measures supported on periodic orbits have the same Lyapunov exponents. 

\begin{question}
For a continuous cocycles over hyperbolic systems, does constant periodic data imply that the cocycle is monochromatic?
\end{question}

The author expects a negative answer. 
We note that Kalinin's periodic approximation theorem \cite[Theorem~1.4]{Kalinin} fails in the continuous class, as shown in \cite{Bochi_isolated}.

\medskip

The monochromatic cocycles constructed in this paper take values in the disconnected Lie group $\mathrm{SL}^\smallpm(2,\R)$. If we want examples taking values in a connected Lie group, we can use the fact that $\mathrm{SL}^\smallpm(2,\R)$ embeds into $\mathrm{SL}(d,\R)$, for any $d > 2$. 

We can also construct \emph{some} examples taking values in $\mathrm{SL}(2,\R)$, for instance:

\begin{proposition}
If $T = \sigma_4 \colon 4^\Z \to 4^\Z$ is the full shift on four symbols, then for any $\lambda_0>0$ we can find a map $A \colon 4^\Z \to \mathrm{SL}(2,\R)$ satisfying the conclusions of  \cref{t.main_SFT}.
\end{proposition}

\begin{proof}
Let $\sigma_2$ be the full shift on two symbols. 
By \cref{t.main_SFT}, there exists $B \colon 2^\Z \to \mathrm{SL}^\smallpm(2,\R)$ such that $\lambda(\sigma_2,B,\mu) = \lambda_0/2$ for every $\sigma_2$-invariant ergodic measure $\mu$ and the cocycle $(\sigma_2,B)$ is not uniformly hyperbolic. The proof of the theorem shows that $B$ can be taken with constant  determinant $-1$. In particular, $B^{(2)}(x) = B(\sigma_2(x)) B(x)$ has determinant $1$. Note that $\sigma_2$ is a square root of $\sigma_4$; more precisely, there exists a homeomorphism $h \colon 4^\Z \to 2^\Z$ such that $h \circ \sigma_4 = \sigma_2^2 \circ h$. Therefore $A \coloneqq B^{(2)} \circ h$ has the required properties. 
\end{proof}

We fundamentally used the existence of a square root. On the other hand, $\sigma_2$ admits no square root, since it has a unique orbit of period $2$. Thus we ask:

\begin{question}
If $T = \sigma_2$ is the full shift on two symbols and $\lambda_0>0$, can we find a map $A \colon 2^\Z \to \mathrm{SL}(2,\R)$ satisfying the conclusions of \cref{t.main_SFT}?
\end{question}

\medskip

As indicated in the introduction, the ``uniformly nonuniformly hyperbolic'' example of Velozo Ruiz \cite[Theorem~4.1]{Velozo} was preceded by examples of \emph{derivative cocycles} with the same property: see \cite{CLR,Gogolev}. In fact, the problem of under which conditions the absence of uniform hyperbolicity for a diffeomorphism implies the existence of a non-hyperbolic measure has been extensively studied over the past two decades: see \cite{DYZ} and the references therein. We close this paper by asking whether monochromatic nonuniform hyperbolicity can be found in smooth dynamics:

\begin{question}
Does there exist a $C^1$-diffeomorphism $f$ of a compact smooth manifold $M$ admitting a non-hyperbolic homoclinic class $H$ such that the derivative cocycle restricted to $H$ is monochromatic?
\end{question}

By a theorem of \cite{CCGWY}, such examples form a meager set in the space $\mathrm{Diff}^1(M)$.


\bigskip

\noindent \textbf{Acknowledgements.} 
I thank Omri Sarig and Scott Schmieding for invaluable help with the construction of a Veech-like map, and Jon DeWitt and Andrey Gogolev for useful comments. 
I am grateful to the referees for their careful reading of the paper and many suggestions to improve the exposition. 


\end{document}